\documentclass{amsart}
\usepackage[utf8]{inputenc}
\usepackage{amsmath,amssymb,amsthm,amsfonts}
\usepackage{xcolor}
\usepackage{mathrsfs}
\usepackage{verbatim}
\usepackage{mathtools}
\usepackage{float}
\usepackage{dsfont}
\usepackage{enumerate}
\usepackage[colorlinks, citecolor=blue, linkcolor=red]{hyperref}
\usepackage{dutchcal}
\allowdisplaybreaks

\numberwithin{equation}{section}
\numberwithin{table}{section}
\numberwithin{figure}{section}

\newtheorem{satz}{Theorem}[section]
\newtheorem{theorem}[satz]{Theorem}
\newtheorem{lemma}[satz]{Lemma}
\newtheorem{korollar}[satz]{Corollary}
\newtheorem{proposition}[satz]{Proposition}

\newtheorem{hyp}[satz]{Hypothesis}
\theoremstyle{definition}
\newtheorem{definition}[satz]{Definition}
\newtheorem{bemerkung}[satz]{Remark}

\renewcommand{\a}{{\mathfrak{a}}}

\DeclareMathOperator{\re}{Re}
\DeclareMathOperator{\im}{Im}

\newcommand{\Hh}{{\mathcal{H}}}
\newcommand{\uu}{\mathcal{u}}
\newcommand{\vv}{\mathcal{v}}
\newcommand{\Aa}{\mathcal{A}}
\newcommand{\Bb}{\mathcal{B}}
\newcommand{\Tt}{\mathcal{T}}
\newcommand{\ff}{\mathcal{f}}
\newcommand{\ee}{\mathcal{e}}

\newcommand{\dx}{\mathrm{d}}
\newcommand{\tr}{\mathop{\mathrm{tr}}}
\renewcommand{\ker}{\mathop{\mathrm{ker}}}
\newcommand{\<}{\left\langle}
\renewcommand{\>}{\right\rangle}
\newcommand{\la}{\langle}
\newcommand{\ra}{\rangle}
\newcommand{\one}{{\mathds{1}}}
\renewcommand{\phi}{\varphi}
\newcommand{\eps}{\varepsilon}

\newcommand{\R}{\mathbb{R}}
\DeclareMathOperator{\essinf}{ess\,inf}

\begin{document}
\title[The Bi-Laplacian with Wentzell conditions on Lipschitz domains]{The Bi-Laplacian with Wentzell boundary conditions on Lipschitz domains}

\author{Robert Denk}
\address{R.\ Denk, Universit\"at Konstanz, Fachbereich f\"ur Mathematik und Statistik, Konstanz, Germany}
\email{robert.denk@uni-konstanz.de}

\author{Markus Kunze}
\address{M.\ Kunze, Universit\"at Konstanz, Fachbereich f\"ur Mathematik und Statistik, Konstanz, Germany}
\email{markus.kunze@uni-konstanz.de}
\author{David Plo\ss}
\address{D.\ Plo\ss, Universit\"at Konstanz, Fachbereich f\"ur Mathematik und Statistik, Konstanz, Germany}
\email{david.ploss@uni-konstanz.de}

\begin{abstract}
We investigate the Bi-Laplacian with Wentzell boundary conditions in a bounded domain $\Omega\subseteq\R^d$
with Lipschitz boundary $\Gamma$. More precisely, using form methods, we show that the associated operator on the ground space $L^2(\Omega)\times L^2(\Gamma)$ has compact resolvent and generates a holomorphic and strongly continuous real semigroup of self-adjoint operators. Furthermore, we give a full characterization of the domain in terms of Sobolev spaces, also proving H\"older regularity of solutions, allowing classical interpretation of the boundary condition. Finally, we investigate spectrum and asymptotic behavior of the semigroup, as well as eventual positivity.
\end{abstract}

\date{December 21, 2020}

\maketitle

\section{Introduction}

 Wentzell or dynamic boundary conditions appear naturally in many physical contexts where a free energy on the boundary of the domain has to be taken into account. This is the case, for instance, for the heat equation with heat sources on the boundary (see \cite[Section~3]{Gol06}),  for the Stefan problem with surface tension (see \cite[Section~1]{Escher-Pruess-Simonett03}), in climate models including  coupling between the deep ocean and the surface (see \cite[Section~2]{DT08}), and for the Cahn--Hilliard equation describing spinodal decomposition of binary polymer mixtures (see \cite[Section~1]{Racke-Zheng03}).
From a mathematical point of view, the fact that the time derivative
of the unknown function appears on the boundary implies that  classical parabolic theory cannot be applied.
Therefore, new methods (mostly based on semigroup theory) were developed for boundary value problems with
Wentzell boundary condition (see, e.g., \cite{AMPR03}, \cite{EF05}, \cite{War13}). Most of these results deal
with the Laplacian or more general second-order operators. For the Bi-Laplacian with Wentzell boundary conditions, less
results are known, and typically the smooth setting is considered (see \cite{FGGR08}). Therefore, it is an interesting
task to study the Bi-Laplacian with Wentzell boundary condition in a bounded  domain $\Omega$ with Lipschitz boundary $\Gamma$. This is the topic of the present paper.

The main challenge in tackling Wentzell boundary conditions lies in the fact that the operator  of the equation in the interior, in our case the Bi-Laplacian $\Delta^2$, itself appears in the boundary condition,
and the standard condition $\Delta^2u \in L^2(\Omega)$ is not sufficient to guarantee existence of the trace of $\Delta^2u$ on the boundary. The most common way to solve this problem is
 to consider a related operator in the product space for which the action in the interior of the domain and on the boundary is decoupled.

The case of the Laplace operator subject to Wentzell boundary conditions on Lipschitz domains was treated in this way by
form methods on the space $L^2(\Omega)\times L^2(\Gamma)$ in \cite{AMPR03}; using the classical Beurling--Deny criteria this result is then extended  to the $L^p$-scale. Under additional smoothness assumptions also spaces of continuous functions were considered in \cite{AMPR03}; see also \cite{EF05}, where Greiner perturbations were used.
These results were later extended to general second-order elliptic operators on Lipschitz domains,
see \cite{Nit11} and \cite{War13}.

For higher order elliptic operators the above extension procedure does not work, because the Beurling--Deny criteria are in general not fulfilled (see also Proposition \ref{nonpos} below). An exception is the one-dimensional situation, where one can extend at least to part of the $L^p$-scale, see \cite{gm20a, gm20b}, where fourth order (or even higher order) operators on networks with various boundary and transmission condition for the nodes were studied.

In higher dimensions, less results are available and they typically rely on being in a smooth setting. For fourth-order equations with sufficiently smooth coefficients in $C^4$-domains, it was shown in \cite[Theorem~2.1]{FGGR08} that the related operator in the product space is essentially self-adjoint. For the Cahn--Hilliard equation, classical well-posedness was shown in \cite[Theorem~5.1]{Racke-Zheng03} in the $L^2$-setting, and in \cite[Theorem~2.1]{Pruess-Racke-Zheng06} in the $L^p$-setting. These results were generalized to boundary value problems of relaxation type (including dynamic boundary conditions) in \cite[Theorem~2.1]{Denk-Pruess-Zacher08}, where maximal regularity in $L^p$-spaces is shown. Again the domain and the coefficients were assumed to be (sufficiently) smooth, and the methods do not carry over to the Lipschitz case considered here.

The aim of our paper is to study the evolution equation for a fourth-order operator on a Lipschitz domain with
Wentzell boundary conditions, showing existence of a holomorphic semigroup and giving a full characterization of the domain in terms of Sobolev regularity. More precisely, we consider the initial boundary value problem
\begin{alignat}{4}
  \partial_t u   +\Delta(\alpha  \Delta)u & =0 && \text{ in }   (0,\infty)\times \Omega,\label{1-1}\\
  \Delta(\alpha  \Delta)u +\beta  \partial_\nu (\alpha  \Delta) u -\gamma  u   & = 0
   && \text{ on }  (0,\infty)\times \Gamma,\label{1-2}\\
  \partial_\nu u  & = 0  && \text{ on }   (0,\infty)\times \Gamma,\label{1-3}\\
  u |_{t=0} & = u_0   && \text{ in } \Omega.\label{1-4}
\end{alignat}

In \eqref{1-1}--\eqref{1-4}, it is implicitly assumed that the initial value $u_0$ is sufficiently smooth to have a trace on the boundary and that this trace is used as an initial condition for $u$ on the boundary.

Here, and throughout this article, we make the following assumptions.

\begin{hyp}
The set $\Omega \subseteq \R^d$ is a bounded domain with Lipschitz boundary $\Gamma$. We endow $\Omega$
with Lebesgue measure and $\Gamma$ with surface measure $S$. Moreover, we are given functions
$\alpha \in L^\infty(\Omega; \R)$ and $\beta, \gamma \in L^\infty(\Gamma;\R)$ such that there exists a constant $\eta>0$ with $\alpha \geq \eta$ almost everywhere on $\Omega$ and $\beta \geq \eta$ almost everywhere on $\Gamma$.
\end{hyp}

Note that Equation \eqref{1-1} is of fourth order with respect to $x\in\Omega$, whence we have to impose two boundary conditions. Here, we have chosen the Neumann boundary condition \eqref{1-3} in addition to the Wentzell boundary condition \eqref{1-2}. From \eqref{1-1} we get $\Delta(\alpha\Delta)u =-\partial_t u$, and replacing this  in \eqref{1-2}, we obtain a dynamic boundary condition.

In order to decouple this system as mentioned above, we rename $u$ to $u_1$ and replace in the boundary condition \eqref{1-2} the term $\Delta(\alpha\Delta u)$ not by the time derivative $\partial_t u_1$ but by the time derivative $\partial_t u_2$ of an independent function $u_2$ that lives on the boundary. Even though $u_2$ is formally independent of $u_1$, we think of $u_2$ as the trace of $u_1$; this condition will actually be incorporated into the domain of our operator.  We thus obtain the following decoupled version of \eqref{1-1}--\eqref{1-4}:
\begin{alignat}{4}
  \partial_t u_1   +\Delta(\alpha  \Delta)u_1 & =0 && \text{ in }   (0,\infty)\times \Omega,\label{1-5}\\
  \partial_t u_2 -\beta  \partial_\nu (\alpha  \Delta) u_1 +\gamma  u_2   & = 0
   && \text{ on }  (0,\infty)\times \Gamma,\label{1-6}\\
  \partial_\nu u_1  & = 0  && \text{ on }   (0,\infty)\times \Gamma,\label{1-7}\\
  u_1 |_{t=0} & = u_{1,0}  && \text{ in } \Omega,\label{1-8}\\
  u_2 |_{t=0} & = u_{2,0} && \text{ on } \Gamma.\label{1-9}
\end{alignat}
Note that as $u_2$ is independent of $u_1$, we have to impose an additional initial condition for $u_2$. If, however, the initial value $u_0$ in \eqref{1-4} is smooth enough, we can put $u_{1,0}=u_0$ and $u_{2,0}=u_0|_\Gamma$.

As we are in the situation of a Lipschitz domain, we are outside the usual ‘strong setting’ for differential operators, and we have to define the operator $\mathcal A$ related to \eqref{1-5}--\eqref{1-9} in a weak sense. In the Lipschitz case, the domain of the Neumann Laplacian is, in general, not contained in the Sobolev space $H^2(\Omega)$ and thus the standard Green's Formula is not at our disposal. Therefore,  we use weaker definitions of the Neumann Laplacian and for the Dirichlet and Neumann traces of functions involved. Based on results in  \cite{GM08}, \cite{GM11}, and \cite[Section 8.7]{BHS20}, we establish in Section~2 a version of Green's formula and a regularity result for functions satisfying Green's formula which appears to be new and might be of independent interest, see Proposition~\ref{LemGreen} below.

These results are used in Section~3 to define a quadratic form $\mathfrak a$ (related to the system \eqref{1-5}--\eqref{1-9}) to which the operator $\mathcal A$ is associated. Based on the analysis of the form, we can show that the operator $\mathcal A$ is self-adjoint and the generator of a strongly continuous and analytic semigroup $(\mathcal T(t))_{t\ge 0}$ (Theorem~\ref{t.Agenerator}). However, this semigroup is neither positive nor $L^\infty$-contractive (Proposition~\ref{nonpos}).

In Section~4,  Theorem~\ref{main}, we identify the operator $\mathcal A$ associated to the form $\mathfrak a$ as an operator matrix acting on the product space $L^2(\Omega)\times L^2(\Gamma)$; we also obtain an explicit description of the domain $D(\Aa)$. This will show that the operator $\mathcal{A}$ indeed governs the system \eqref{1-5}--\eqref{1-9}. We will explain  afterwards that we can obtain a solution of the system \eqref{1-1}--\eqref{1-3} with initial condition \eqref{1-4}. If $u_{2,0}$ is not the trace of $u_{1,0}$, there are some subtleties concerning the initial values, see Remark \ref{r.explain}.

One of the main results of this paper, Theorem~\ref{HoelReg} in Section~5, states that for every element  $(u_1, u_2)$ of $D(\mathcal A^\infty)$ the function $u_1$ is H\"older continuous and $u_2$ is the trace of $u_1$. As the semigroup $\Tt$ is analytic, it follows that for positive time the solution of \eqref{1-5}--\eqref{1-9} is H\"older continuous and satisfies the  Wentzell boundary condition in a  pointwise sense. But this regularity result is also of independent interest as $D(\Aa^\infty)$ is a core for $\Aa$ (and also a form core for $\a$, see the proof of
\cite[Lemma 1.25]{Ouh09}). Moreover, this result implies regularity of the eigenfunctions of the operator $\Aa$ and is used extensively in the subsequent sections.

In Section 6, we show that the operator $\mathcal{A}$ has compact resolvent. By standard theory, we thus find an orthonormal basis consisting of eigenfunctions of $\mathcal{A}$. This allows us to describe the semigroup in terms of the  eigenfunctions and study the asymptotic behavior of the semigroup.

In the concluding Section 7, we study eventual positivity of the semigroup. We have already mentioned that our operator does not satisfy the Beurling--Deny criteria. In fact, \cite[Thm.\ 3.6]{MO86} (which is concerned with operators on $\R^d)$ suggests that a semigroup generated by a fourth-order operator cannot be expected to be positive; similar results have also been observed for the Bi-Laplacian subject to Dirichlet boundary conditions, see \cite[Sections 3.1.3 and 5.1]{Gazzola-Grunau-Sweers10}. However, for \emph{some} domains $\Omega$ the semigroup generated by the Bi-Laplacian with Dirichlet boundary conditions is at least, in a sense, ``eventually positive''.
We will see that for $\gamma\equiv 0$ and independently of the geometry of $\Omega$ this is also true for our semigroup (Theorem \ref{EventPos}). If, however, $\gamma>0$, then, similar to Dirichlet boundary conditions, there are domains where eventual positivity fails, see Corollary \ref{sign}.

\section{The Neumann Laplacian  and Green's formula on Lipschitz domains}\label{trace}

As we consider a fourth-order equation in a Lipschitz domain, the definition of the operator related to \eqref{1-1}--\eqref{1-4} in Section~3 will be based on the related quadratic form, so we are in the weak setting. To handle this situation, we start with the (weakly defined) Neumann Laplacian which is the topic of the present section. Weak traces and the Dirichlet and Neumann Laplacian in Lipschitz domains were studied, e.g., in  \cite{GM08}, \cite{GM11}, and \cite{BHS20}.

For $s\geq 0$, we write $H^s(\Omega)$ for the standard Sobolev space and $H^s_\Delta(\Omega)$ for the space of  functions $u\in H^s(\Omega)$ such that the distributional Laplacian $\Delta u$ belongs to $L^2(\Omega)$. We denote the inner products in  $L^2(\Omega)$ and $L^2(\Gamma)$ by
\[
\la f,g\ra_\Omega \coloneqq \int_\Omega u\overline{v}\, \dx  x\quad \mbox{and} \quad \la f, g\ra_\Gamma \coloneqq \int_\Gamma f\overline{g}\, \dx S
\]
respectively. By slight abuse of notation, we will also write
\[
\la \nabla u, \nabla v\ra_\Omega\coloneqq \int_\Omega \sum_{j=1}^d \partial_j u\overline{\partial_j v}\, \dx x
\]
whenever $u, v\in H^1(\Omega)$. We write $\|\cdot\|_\Omega$ and $\|\cdot\|_\Gamma$  for the induced norms. In $H^s_\Delta(\Omega)$, we take the canonical norm
\[ \|u\|_{H^s_\Delta(\Omega)}^2 \coloneqq \|u\|_{H^s(\Omega)}^2 + \|\Delta u\|_\Omega^2,\quad u\in H^s_\Delta(\Omega).\]
We write $H^s(\Gamma)$, $s\in [-1,1]$, for the standard Sobolev spaces on the Lipschitz boundary $\Gamma$ (see, e.g., \cite[p.~96]{McLean00}).

The \emph{Neumann Laplacian} $\Delta_N$ on $\Omega$ can now be defined by setting
\begin{equation}\label{2-1}
D(\Delta_N)\coloneqq \{ u \in H^1_\Delta(\Omega) \,|\, \la \nabla u, \nabla v\ra_\Omega = -\la \Delta u, v\ra_\Omega
\text{ for all } v\in H^1(\Omega)\}
\end{equation}
and $\Delta_N u = \Delta u$, the distributional Laplacian.

To  describe in which sense elements of $D(\Delta_N)$ satisfy Neumann boundary conditions, one has to study (weak) traces on the boundary. Let $C_c^\infty(\R^d)$ denote the space of all infinitely smooth functions on $\R^d$ with compact support, and let $C^\infty(\overline{\Omega}):=\{\phi|_\Omega \,|\, \phi \in C_c^{\infty}(\R^d)\}$. We denote the trace of a function $u\in C^\infty(\overline\Omega)$ on the boundary by $\tr u := u|_{\Gamma}$. This \emph{smooth trace} extends by continuity to a bounded linear operator $\tr\colon H^s(\Omega)\to H^{s-1/2}(\Gamma)$ for all $s\in (\frac12, \frac32)$ \cite[Theorem~3.38]{McLean00}. For $s\in (\frac12,1]$, this operator is surjective and  even a retraction, i.e. there exists a continuous right-inverse (see \cite[Theorem~3.37]{McLean00}).

Even for smooth domains, the continuity of $\tr\colon H^s(\Omega)\to H^{s-1/2}(\Gamma)$ does not hold for the endpoint case $s=\frac12$, see \cite[Theorem~1.9.5]{LM72}. However, one can include the cases $s=\frac12$ and $s=\frac32$ by considering the spaces $H^s_\Delta(\Omega)$ instead of $H^s(\Omega)$. It was shown in \cite[Lemma~2.3]{GM08} that the smooth trace extends to a retraction $\tau_D\colon H^{3/2}_\Delta(\Omega)\to H^1(\Gamma)$. Similarly, we can consider the\emph{ smooth Neumann trace} $u\mapsto \nu\cdot \tr(\nabla u),\, u\in C^\infty(\overline \Omega)$, where $\nu$ denotes the unit outer normal which exists in almost every boundary point. This trace extends to a retraction $\tau_N\colon H^{3/2}_\Delta(\Omega)\to L^2(\Gamma)$, see \cite[Lemma~2.4]{GM08}.

For the connection between the above traces and the Neumann Laplacian, we consider the \emph{weak Neumann trace} $\partial_\nu$ which is defined on
\begin{align*}
D(\partial_\nu) \coloneqq  \big\{ u \in H^1_\Delta(\Omega) \, &| \, \mbox{there exists } g\in L^2(\Gamma) \mbox{ such that}\\
& \quad\la \Delta u, v\ra_\Omega +\la \nabla u, \nabla v\ra_\Omega = \la g, \tr v\ra_\Gamma \text{ for all } v\in H^1(\Omega)\big\}
\end{align*}
by setting $\partial_\nu u = g$. As $\tr\colon H^1(\Omega)\to H^{1/2}(\Gamma)$ is surjective and $H^{1/2}(\Gamma)$ is dense in $L^2(\Gamma)$ (cf. \cite[Section~8.7]{BHS20}), the function $g\in L^2(\Gamma)$ is unique, which shows that $\partial_\nu u$ is well defined. Thus, it follows that
\[
D(\Delta_N) = \{ u\in H^1_\Delta(\Omega) \,|\, \partial_\nu u =0\}.
\]

\begin{bemerkung}
We would like to point out that the definition of the smooth Neumann trace $\tau_N$ (though not that of its extension to $H^{3/2}_\Delta(\Omega)$) depends only on the geometry of the domain and is independent of the choice of the underlying operator, in our case the Laplacian. The weak Neumann trace, on the other hand, depends crucially on the fact that we consider the Laplacian. If, instead, we consider a general second order elliptic differential operator $A$ in divergence form, we would instead obtain the co-normal derivative $\partial_\nu^A$ associated to $A$. It would be more appropriate to use the notation $\partial^\Delta_\nu$ to indicate the dependence on the underlying operator. However, to simplify notation, we will simply use $\partial_\nu$ as above.
\end{bemerkung}

The following result  shows the connection between the weak Neumann trace and $\tau_N$ and includes  a regularity result for the weak Neumann Laplacian defined above. It can be found in \cite[Theorem\ 8.7.2]{BHS20}.

\begin{lemma}\label{trace2}
We have  $D(\Delta_N)=\{u \in H^{3/2}_\Delta(\Omega)\, |\, \tau_N u=0\}$.
\end{lemma}

Following \cite[Chapter 8]{BHS20}, it is possible to extend the trace operators $\tau_N$ and $\tau_D$ to the space
$H^0_\Delta(\Omega) \coloneqq \{ u\in L^2(\Omega) \,|\, \Delta u \in L^2(\Omega)\}$. The price to pay is that
the extensions take values in certain spaces of functionals on the boundary. This involves the spaces
\[
\mathcal{G}_0\coloneqq \mathrm{rg}(\tau_D|_{\mathrm{ker (\tau_N)}})\quad \mbox{and}\quad \mathcal{G}_1 \coloneqq \mathrm{rg}(\tau_N|_{\mathrm{ker (\tau_D)}}),
\]
where $\mathrm{rg}$ stands for the range of an operator. It is possible to define a  Hilbert space structure on those spaces creating two Gelfand triples $\mathcal{G}_0 \subseteq L^2(\Gamma) \subseteq \mathcal{G}_0'$ and $\mathcal{G}_1 \subseteq L^2(\Gamma) \subseteq \mathcal{G}_1'$. We  recall the following result from \cite[Theorem\ 8.7.5]{BHS20}.

\begin{lemma}\label{ExTrace}
The traces $\tau_D$ and $\tau_N$ can be continuously extended to bounded linear operators
\[
\tilde \tau_D: H^0_\Delta(\Omega) \rightarrow \mathcal{G}_1'\quad\mbox{and}\quad \tilde \tau_N: H^0_\Delta(\Omega) \rightarrow \mathcal{G}_0',
\]
respectively. Moreover,
\begin{enumerate}
[(i)]
\item $\ker \tilde\tau_N= \ker \tau_N = D(\Delta_N)$,
\item  for $u \in H^0_\Delta(\Omega)$ and $v \in D(\Delta_N)$ we have
\begin{equation}\label{eq.greenveryweak}
\la \Delta u, v\ra_\Omega - \la u, \Delta v\ra_\Omega = \la \tilde\tau_N u, \tau_D v\ra_{\mathcal{G}_0'\times\mathcal{G}_0}.
\end{equation}
\end{enumerate}
\end{lemma}

We can now establish a version of Green's formula on Lipschitz domains and obtain regularity for all functions in $H^0_\Delta(\Omega)$ which satisfy this Green's formula. This is the main result of this section. It is worthwhile  to point out that while the extended traces $\tilde\tau_N$ and $\tilde\tau_D$ do not appear in the statement of the result, we make extensive use of them in the proof. Indeed, by virtue of Lemma~\ref{ExTrace}, we can give meaning to traces of functions in $H^0_\Delta(\Omega)$ and have \eqref{eq.greenveryweak} at our disposal. We may then use the fact that $\ker \tau_N = \ker \tilde\tau_N$ to infer higher regularity of the functions involved.

\begin{proposition}\label{LemGreen}~
\begin{enumerate}[(i)]
\item We have  $\partial_\nu = \tau_N$ and, in particular, $D(\partial_\nu) = H^{3/2}_\Delta(\Omega)$.
For $u \in D(\partial_\nu)$ and $v \in D(\Delta_N)$, we have
\begin{equation}\label{2-2}
\la \Delta u, v\ra_\Omega - \la u, \Delta v\ra_\Omega = \la \partial_\nu u, \tr v\ra_\Gamma.
\end{equation}
\item Let $u \in H^0_\Delta(\Omega)$ and assume there is some $g \in L^2(\Gamma)$ such that for all $v \in D(\Delta_N)$
we have
\begin{align}\label{GreenEx}
\<\Delta u,v\>_\Omega-\<u, \Delta v\>_\Omega=\<g, \tr v\>_\Gamma.
\end{align}
Then $u \in H^{3/2}_\Delta(\Omega)$ and $\partial_\nu u=g$.
\end{enumerate}
\end{proposition}

\begin{proof}
(i) Fix $u \in H^{3/2}_\Delta (\Omega)\subseteq H^1_\Delta(\Omega)$ and let $v\in D(\Delta_N)$. Noting that  $\tr v = \tau_D v \in \mathcal{G}_0$, Equality \eqref{eq.greenveryweak} yields
\[
\<\Delta u,v\>_\Omega-\<u, \Delta v\>_\Omega=\<\tau_N u, \tr v\>_{\mathcal{G}_0' \times \mathcal{G}_0}.
\]
As $\tau_N u \in L^2(\Gamma)$ and $\mathcal G_0\subseteq L^2(\Gamma)\subseteq \mathcal G_0'$ is a  Gelfand triple, we obtain
$$\<\tau_N u, \tr v\>_{\mathcal{G}_0' \times \mathcal{G}_0}=\<\tau_N u, \tr v\>_\Gamma.$$
Consequently,
\begin{align}\label{greentn}
\<\Delta u,v\>_\Omega-\<u, \Delta v\>_\Omega=\<\tau_N u, \tr v\>_\Gamma.
\end{align}
Since $v \in D(\Delta_N)$ and $u \in H^1(\Omega)$, we have $\<u,\Delta v\>=-\<\nabla u, \nabla v\>$ by \eqref{2-1}, and thus \eqref{greentn} can be rewritten as
\begin{equation}\label{eq.neumanntrace}
\<\Delta u,v\>_\Omega+\<\nabla u, \nabla v\>_\Omega=\<\tau_N u, \tr v\>_\Gamma.
\end{equation}
Note that $ \Delta_N$ is the associated operator of the closed symmetric form $(u,v)\mapsto\langle \nabla u,\nabla v\rangle_\Omega$ with form domain $H^1(\Omega)$. Thus,
by \cite[Lemma 1.25]{Ouh09}, $D(\Delta_N)$ is dense in $H^1(\Omega)$. As, moreover, $\tr$ is a continuous map from
$H^1(\Omega)$ to $L^2(\Omega)$, we can extend \eqref{eq.neumanntrace} by density to hold for all $v \in H^1(\Omega)$.
It follows that $u \in D(\partial_\nu)$ and $\partial_\nu u=\tau_N u$, which proves  $\tau_N \subseteq \partial_\nu$.

It remains to show that $D(\partial_\nu)\subseteq D(\tau_N)= H^{3/2}_\Delta(\Omega)$. For this,  let $u \in D(\partial_\nu)$ and set $g:= \partial_\nu u$. Then for $v \in D(\Delta_N)$ we have, by definition of $\partial_\nu$,
\begin{equation}\label{2-3}
 \<\Delta u,v\>_\Omega+\<\nabla u, \nabla v\>_\Omega = \<g, \tr v\>_\Gamma = \<g, \tr v\>_{\mathcal{G}_0' \times \mathcal{G}_0},
\end{equation}
where the second equality holds since $\tr v\in \mathcal G_0$. As $v\in D(\Delta_N)$ and $u\in H^1(\Omega)$, we have $\< u,\Delta v\>_\Omega = -\< \nabla u,\nabla v\>_\Omega$ by \eqref{2-1}, and we obtain
\begin{equation}
  \label{2-4}
   \<\Delta u,v\>_\Omega+\<\nabla u, \nabla v\>_\Omega = \<\Delta u,v\>_\Omega-\<u, \Delta v\>_\Omega=\<\tilde\tau_N u, \tr v\>_{\mathcal{G}_0' \times \mathcal{G}_0}.
\end{equation}
A comparison of \eqref{2-3} and \eqref{2-4} shows that $\<g, \tr v\>_{\mathcal{G}_0' \times \mathcal{G}_0}= \<\tilde\tau_N u, \tr v\>_{\mathcal{G}_0' \times \mathcal{G}_0}$.
As $v \in D(\Delta_N)=\ker \tau_N$ was arbitrary, we have  $g=\tilde\tau_N u$ in $\mathcal G_0'$. Since $g \in L^2(\Gamma)$ and $\tau_N$ is surjective,
 we can find a function $\bar u \in H^{3/2}_\Delta(\Omega)$ such that $\tau_N \bar u=g=\tilde\tau_N u.$ Hence $\bar u-u\in \ker \tilde\tau_N=\ker \tau_N\subseteq H^{3/2}_\Delta(\Omega)$. But then also $u=\bar u-(\bar u-u) \in H^{3/2}_\Delta(\Omega)$, which shows $D(\partial_\nu)\subseteq D(\tau_N)$ and, consequently, $\partial_\nu = \tau_N$.
Equality \eqref{2-2} is now an immediate consequence of \eqref{greentn}.

\smallskip

(ii) Here we may argue in a similar way as in the proof of (i).
Let $u \in H^0_\Delta(\Omega)$ and $g \in L^2(\Gamma)$ such that for all $v \in D(\Delta_N)$ we  have
\[
\<\Delta u,v\>-\<u, \Delta v\>=\<g, \tr v\>_\Gamma.
\]
Comparing with \eqref{eq.greenveryweak}, we obtain
\[
\<g, \tr v\>_{\mathcal{G}_0' \times \mathcal{G}_0} =\<\tilde\tau_N u, \tr v\>_{\mathcal{G}_0' \times \mathcal{G}_0},
\]
for all $v\in D(\Delta_N)$ and thus $\tilde\tau_N u=g$. Making use of the surjectivity of $\tau_N$ and the fact that $\ker\tilde\tau_N\subseteq H^{3/2}_\Delta(\Omega)$, the same arguments as before yield $u \in H^{3/2}_\Delta(\Omega)$ and $\partial_\nu u=\tau_N u=g$.
\end{proof}

\section{The Bi-Laplacian via quadratic forms}

We now take up our main line of study and define a quadratic form which will then be used to define a realization of the Bi-Laplace operator. In contrast to the last section, we  now combine the $L^2$-spaces on $\Omega$ and on $\Gamma$ into a single Hilbert space. Moreover, we will incorporate the function $\beta$ into its norm. More precisely, we set
\[
\Hh \coloneqq L^2(\Omega) \times L^2(\Gamma, \beta^{-1}\dx S),
\]
where the inner product on the second factor is given by
\[
\la u,v\ra_{\Gamma,\beta}\coloneqq \int_\Gamma u\bar v \beta^{-1}\, \dx S.
\]
To be consistent with the last section, we will omit the subscript $\beta$ when $\beta = \one$ is the constant one function:
$\la \cdot, \cdot\ra_{\Gamma, \one} = \la \cdot, \cdot\ra_\Gamma$. Note that as $\beta, \beta^{-1} \in L^\infty(\Gamma)$ the scalar products $\la \cdot, \cdot\ra_{\Gamma, \beta}$ and $\la \cdot, \cdot\ra_\Gamma$ are always equivalent.

We will denote elements of $\Hh$ by lowercase calligraphic letters and the components of this element by the same lowercase roman letters, i.e.\ if $\uu, \vv \in \Hh$, then $\uu = (u_1, u_2)$, $\vv = (v_1, v_2)$ and
\[
\la \uu, \vv\ra_\Hh = \la u_1, v_1\ra_\Omega + \la u_2, v_2 \ra_{\Gamma, \beta}.
\]

We may now define our quadratic form. For general information concerning forms and their associated operators we refer the reader to \cite[Chapter 6]{Kat95} or \cite{Ouh09}.

\begin{definition} We define the form $\a$ by setting
\begin{align*}
\a (\uu, \vv) & \coloneqq \int_\Omega \alpha \Delta u_1 \overline{\Delta v_1}\, dx  + \int_\Gamma \gamma  u_2\overline{v_2}\beta^{-1}\, \dx S\\
& = \la \alpha \Delta u_1, \Delta v_1\ra_\Omega + \la \gamma u_2, v_2\ra_{\Gamma, \beta}
\end{align*}
for
\[
\uu, \vv \in D(\a) \coloneqq \{ u\in \Hh\, |\,  u_1 \in D(\Delta_N), u_2 = \tr u_1\}.
\]
\end{definition}

\begin{lemma}\label{dd}
The form domain $D(\a)$ is a dense subset of $\Hh$.
\end{lemma}

\begin{proof}
We may assume without loss of generality that $\beta = \one$, otherwise switching to an equivalent norm.
Next note that $C_c^\infty(\Omega)\times\{0\} \subseteq D(\a)$ which implies that $L^2(\Omega)\times\{0\} \subseteq
\overline{D(\a)}$ as the test functions $C_c^\infty(\Omega)$ are dense in $L^2(\Omega)$.

We next show that $\{0\} \times L^2(\Gamma) \subseteq \overline{D(\a)}$. To that end, let $f_2 \in L^2(\Gamma)$ and $\eps >0$. As $H^{1/2}(\Gamma)$ is dense in $L^2(\Gamma)$, we find a function $u_2 \in H^{1/2}(\Gamma)$ with $\|u_2-f_2\|_\Gamma^2 \leq \eps$. Because $\tr\colon H^1(\Omega)\to H^{1/2}(\Gamma)$ is bounded  (denote its operator norm by $M$) and surjective (cf.\ Section \ref{trace}), we find a function $\tilde u_1 \in H^1(\Omega)$ with $\tr\tilde u_1 = u_2$. As also $D(\Delta_N)$ is dense in $H^1(\Omega)$, we find a function $\bar u_1 \in D(\Delta_N)$ with $\| \bar u_1 - \tilde u_1\|_{H^1(\Omega)}^2 \leq M^{-1}\eps$. Finally, we pick a test function $\varphi \in C_c^\infty(\Omega)$ such that $\|\bar u_1 - \varphi\|_\Omega^2 \leq \eps$ and put
$\uu = (\bar u_1 - \varphi, \tr (\bar u_1 -\varphi)) = (\bar u_1 - \varphi, \tr \bar u_1)$. Then, by construction,
we have $\uu \in D(\a)$ and a short computation shows $\| \uu - (0, f_2)\|_{\Hh}^2 \leq 3\eps$. As $f_2$ was arbitrary,
$\{0\}\times L^2(\Gamma) \subseteq \overline{D(\a)}$.

Since $\overline{D(\a)}$ is a vector space, we may combine our two results and obtain $\overline{D(\a)} =\Hh$.
\end{proof}

We can now prove the following result.

\begin{proposition}\label{genform}

The form $\a$ is densely defined, symmetric, semibounded from below by $\gamma_0 \coloneqq \min\{\essinf \gamma, 0\}$ (in particular, it is accretive whenever $\gamma \geq 0$), and closed.
\end{proposition}

\begin{proof}
It is straightforward to prove that $\a$ is symmetric, and we have proved that it is densely defined in Lemma \ref{dd}.
For the quadratic form we have
\begin{align*}
\a (\uu) \coloneqq \a (\uu, \uu) & = \int_\Omega \alpha |\Delta u_1|^2\, dx + \int_\Gamma \gamma |u_2|^2\beta^{-1}\, \dx S\\
& \geq \essinf \gamma \cdot \|u_2\|^2_{\Gamma, \beta} \geq \gamma_0\|\uu\|_{\Hh}^2,
\end{align*}
proving the result concerning the semiboundedness. It only remains to prove the closedness.
To that end, we assume without loss of generality that $\gamma\geq 0$ so that
the norm associated with $\a$ on $D(\a)$ is given by $\|\uu\|_\a^2 = \a(\uu) + \|\uu\|_{\Hh}^2$.

Let $(\uu_n)_{n\in \mathbb{N}} \subseteq  D(\a)$ be a $\|\cdot\|_\a$-Cauchy sequence, where $\uu_n = (u_1^n, u_2^n)$.
We have to prove that this sequence converges with respect to $\|\cdot\|_\a$. Let us first note that for a certain constant $C$, we have
\[
\|u_1\|_{\Delta_N}^2 \leq C \|\uu\|_\a^2
\]
whenever $\uu = (u_1, u_2) \in D(\a)$. Here, $\|\cdot\|_{\Delta_N}$ stands for the graph norm of the operator $\Delta_N$. It follows that $u_1^n$ is a Cauchy sequence with respect to $\|\cdot\|_{\Delta_N}$. As  $\Delta_N$ is closed, we find some $u\in D(\Delta_N)$ such that
$u_1^n \to u$ in $L^2(\Omega)$ and $\Delta u_1^n\to \Delta u$ in $L^2(\Omega)$.

Next observe that for $u\in D(\Delta_N)$ we have
\[
\|u\|^2_{H^1(\Omega)} =\|u\|_\Omega^2 + \la \nabla u, \nabla u\ra_\Omega = \|u\|_\Omega^2 - \la \Delta u, u\ra \leq \tilde C(\|\Delta u\|_\Omega^2 + \|u\|_\Omega^2)
\]
for some constant $\tilde C\geq 1$.
Combining this with the above, we find that $u_1^n$ is also a Cauchy sequence in $H^1(\Omega)$ whence, by the continuity of the trace, $u_2^n = \tr u_1^n\to \tr u$ in $L^2(\Gamma)$. Setting $\uu = (u, \tr u)$, we see that
$\uu \in D(\a)$ and $\uu_n \to \uu$ with respect to $\|\cdot\|_\a$.
This proves closedness of the form.
\end{proof}

Proposition \ref{genform} enables us to invoke a representation theorem for semibounded, symmetric forms, see
\cite[Theorem VI.2.6]{Kat95}, to obtain information about the associated operator $\Aa$. We recall that this operator is defined as follows.

The domain $D(\Aa)$ is given by
\begin{equation}\label{DomAss}
D(\Aa) \coloneqq \{\uu \in D(\a) \, |\, \exists\, \ff \in \Hh : \a (\uu, \vv ) = \la \ff , \vv \ra_\Hh \mbox{ for all } \vv\in D(\a)\}
\end{equation}
and for $\uu \in D(\a)$ we have $\Aa \uu = \ff$, where $\ff$ is as in \eqref{DomAss}.

\begin{theorem}\label{t.Agenerator}
The operator $\Aa$ is  self-adjoint  and semibounded.
 Moreover, $-\Aa$ generates a strongly continuous, analytic semigroup $\Tt = (\Tt(t))_{t\geq 0}$ of self-adjoint operators on $\Hh$. If $\gamma \geq 0$, this semigroup is contractive.
\end{theorem}

\begin{proof}
The first statements follow from Proposition \ref{genform} and
the representation theorem \cite[Theorem VI.2.6]{Kat95}. The rest can then either be inferred from the spectral theorem or, else, follows from more general results concerning $m$-sectorial operators, see \cite[Section 1.4]{Ouh09}.
\end{proof}

Up to now, we only have the abstract definition of the operator $\Aa$, given by \eqref{DomAss}, but we will identify this operator more explicitly in the next section. Before we do that, however, we collect some more information about the semigroup $\Tt$. In the study of second-order elliptic operators, defined by means of sectorial forms, contractivity properties of the associated semigroup are of particular importance and can be characterized in terms of the form by means of the Beurling--Deny criteria, see \cite[Chapter 2]{Ouh09}.

Let us briefly recall the relevant notions. To that end, let $(X, \Sigma, \mu)$ be a measure space.
Given a semigroup $(T(t))_{t\geq 0}$ on $H=L^2(X; \mathbb{C})$,
we say that
$T$ is \emph{real} if $T(t)f \in L^2(X; \R)$ for all $t\geq 0$ whenever $f\in L^2(X; \R)$. It is called \emph{positive} if $T(t)f \geq 0$ for all $t\geq 0$ and $f\geq 0$ and \emph{$L^\infty$-contractive} if $\|T(t)f\|_\infty
\leq \|f\|_\infty$ for all $t\geq 0$ and $f\in L^2(X)\cap L^\infty(X)$.  To make use of this terminology in our situation, we use $X= \Omega\cup \Gamma$, $\mu(A) = \lambda(A\cap\Omega) + \int_{A\cap\Gamma}\beta^{-1}\dx S$ and identify our semigroup on the product space $\Hh$ with a semigroup on $L^2(X)$.

We now obtain the following result for our semigroup $\Tt$, in which we restrict ourselves to the situation where
$\gamma\geq 0$, so that $\a$ is accretive.

\begin{proposition}\label{nonpos}
Let $\gamma\geq 0$. Then the semigroup $\Tt$ is real, but neither positive nor $L^\infty$-contractive.
\end{proposition}

\begin{proof}
That $\Tt$ is real can be inferred from \cite[Theorem 2.5]{Ouh09} as $\re D(\a) \subseteq D(\a)$ and $\a (\re \uu, \im \uu) \in \R$ for all $\uu \in D(\a)$.

For $\Tt$ to be positive, it is necessary that $\uu^+ := \sup\{\uu,0\} \in D(\a)$ whenever $\uu \in D(\a)$ is a real-valued function, see \cite[Theorem 2.6]{Ouh09}. But this is never the case. To see this, let us first consider $d=1$ and $\Omega = (-2, 2)$.
We put $\varphi (t) = t\varphi_0(t)$, where $\varphi_0\in C_c^\infty(\Omega)$ with  $\varphi_0=1$ on $[-1,1]$. Then
$\varphi$ belongs to the domain of the Neumann Laplacian (which in this case is $\{u\in H^2(-2, 2) \, | \, u'(-2)=u'(2) = 0\}$). However, if we consider $\varphi^+$, we have $(\varphi^+)' = \one_{(0,1)}$ on the interval $(-1,1)$ and the second derivative is no longer an element of $L^2(\Omega)$, whence $\varphi^+ \not \in H^2(\Omega)$ and thus  $\varphi^+ \not \in D(\Delta_N)$.

This example can be lifted to higher dimensions by considering functions of the form
$u(x_1, \ldots, x_d) = \varphi (s^{-1}(x_1 -c_1))\psi (x_1, \ldots, x_d)$ where $c= (c_1, \ldots, c_d) \in \Omega$, $s>0$ and $\psi$ is a test function which is 1 in a neighborhood of $c$. Then $\uu = (u, 0) \in D(\a)$, but $\uu^+ = (u^+, 0)$ is not.

By \cite[Theorem 2.13]{Ouh09}, for $\Tt$ to be $L^\infty$-contractive, it is necessary that whenever $\uu \in D(\a)$ is a positive, real function, then also $\min\{\uu, 1\}$ belongs to $D(\a)$. But here we can construct a counterexample in a similar way.
\end{proof}

\section{Identification of the associated operator}\label{id}

In this section, we identify the operator associated to our form $\a$, which, in an abstract way, is given by \eqref{DomAss}.
This involves actually two aspects: First, we need to determine the domain of our operator and second, we have to establish how the operator acts on an element of its domain. Since we work in a Hilbert space which is a cartesian product, the action of our operator can be represented by means of an operator matrix. As far as the domain of the operator is concerned, we will give an explicit description in Theorem~\ref{main}. In the smooth setting,  we give an alternative characterization of the domain in Theorem~\ref{HighReg}.
Without additional smoothness assumptions, we obtain the following description of $\mathcal A$. This should be compared to Equations \eqref{1-5} and \eqref{1-6}.

\begin{satz}\label{main}
The operator $\Aa$ associated to the form $\a$ is given by
\[
\Aa=
\begin{pmatrix}
\Delta(\alpha \Delta) & 0 \\
-\beta \partial_\nu (\alpha \Delta) & \gamma
\end{pmatrix},
\]
defined on the domain
\[
D(\Aa)=\big\{\uu \in \Hh : u_1 \in H^{3/2}_\Delta(\Omega), \alpha \Delta u_1 \in H^{3/2}_\Delta(\Omega), \partial_\nu u_1 = 0, u_2 = \tr u_1\big\}.
\]
\end{satz}

We point out that the regularity of an element of $D(\Aa)$ is sufficient for every entry in the above matrix to be well defined as an element of $L^2$. Indeed, as $\alpha \Delta u_1 \in H^{3/2}_\Delta (\Omega)$, it follows that
$\Delta (\alpha \Delta u_1) \in L^2(\Omega)$; moreover, also $\partial_\nu (\alpha \Delta u_1) \in L^2(\Gamma)$,
as $D(\partial_\nu) = H^{3/2}_\Delta(\Omega)$ by Proposition \ref{LemGreen}. Before proceeding to the proof of Theorem~\ref{main}, we collect some alternative characterizations of $D(\Aa)$ for later use.

\begin{korollar}
The domain of the operator $\mathcal A$ is given by
\begin{align*}
D(\Aa)
& =\{ \uu \in \mathcal{H}\, |\, u_1 \in D(\Delta_N), u_2=\tr u_1, \alpha\Delta u \in D(\partial_\nu)\}\\
&=\{\uu \in D(\a)\, |\,  \alpha \Delta u_1 \in D(\partial_\nu)\}.
\end{align*}
\end{korollar}

\begin{proof}
The first equality follows from the fact that $D(\Delta_N) = \{ u\in H^{3/2}_\Delta (\Omega) | \partial_\nu u = 0\}$ (see Lemma \ref{trace2}) and the identity $D(\partial_\nu) = H^{3/2}_\Delta(\Omega)$ from Proposition \ref{LemGreen}. The second is immediate from the definition of $D(\a)$.
\end{proof}

\begin{proof}[Proof of Theorem~{\rm\ref{main}}]
Let us, for time being, denote the operator described in the statement of the theorem by $\Bb$ and by $\Aa$, as before,
the operator associated with the form $\a$. We recall that $\uu \in D(\Aa)$ and $\Aa \uu = \ff$ is equivalent to
$\uu \in D(\a)$ and
\begin{equation}\label{eq.test}
\a (\uu, \vv ) = \la \ff, \vv\ra_\Hh\quad\mbox{ for all }
\vv \in D(\a).
\end{equation}

Let us first prove $\Aa \subseteq \Bb$. To that end, fix $\uu \in D(\Aa)$ and set $\ff := \Aa \uu$. Using
\eqref{eq.test} for $\vv = (\varphi, 0) \in C_c^\infty(\Omega)\times\{0\}\subseteq D(\a)$, we find
\[
\la \alpha \Delta u_1, \Delta \varphi\ra_\Omega = \a (\uu, \vv) = \la \ff, \vv\ra_\Hh = \la f_1, \varphi\ra_\Omega.
\]
As this is true for every $\varphi \in C_c^\infty(\Omega)$, it follows that $f_1 = \Delta (\alpha \Delta u_1)$. In particular,
$\Delta (\alpha \Delta u_1) \in L^2(\Omega)$, so that $\alpha \Delta u_1 \in H^0_\Delta (\Omega)$.

Let $\vv \in D(\a)$. Rearranging terms in \eqref{eq.test} and using that
$\tr v_1 = v_2$ and  $f_1 = \Delta(\alpha \Delta u_1)$, we find
\[
\la \Delta (\alpha \Delta u_1), v_1\ra_\Omega - \la \alpha \Delta u_1, \Delta v_1\ra_\Omega = \la \gamma \tr u_1 - f_2, \tr v_1\ra_{\Gamma, \beta}
\]
for all $v_1\in D(\Delta_N)$. Using Proposition \ref{LemGreen} (ii) with $u=\alpha \Delta u_1$, $v=v_1$ and $g= \beta^{-1}(\gamma \tr u_1 - f_2) \in L^2(\Gamma)$, we obtain $\alpha \Delta u_1 \in H^{3/2}_\Delta(\Omega)$ and $\partial_\nu (\alpha\Delta u_1) = \beta^{-1}(\gamma \tr u_1 - f_2)$. As $u_2=\tr u_1$, the latter is equivalent to $f_2 = -\beta\partial_\nu (\alpha \Delta u_1) + \gamma u_2$. Altogether, we have proved that $\uu \in D(\Bb)$ and $\Aa \uu = \Bb\uu$.\smallskip

To see the converse, let $\uu\in D(\Bb)$. Then, by Lemma \ref{trace2}  and  $\partial_\nu = \tau_N$ (Proposition~\ref{LemGreen} (i)), we find $\uu \in D(\a)$ and $\alpha \Delta u_1 \in H^{3/2}_\Delta (\Omega) = D(\partial_\nu)$.
With \eqref{2-2} we see that for all $\vv \in D(\a)$ we have
\begin{align*}
\a(\uu,\vv)&=\<\alpha\Delta u_1,  \Delta v_1\>_\Omega+\<\gamma u_2,v_2\>_{\Gamma,\beta}\\
&=\<\Delta (\alpha \Delta u_1),  v_1\>_\Omega-\<\partial_\nu (\alpha \Delta u_1),\tr v_1\>_\Gamma+\<\gamma u_2,v_2\>_{\Gamma,\beta}\\
&=\<\Delta (\alpha \Delta u_1),  v_1\>_\Omega+\<-\beta \partial_\nu (\alpha \Delta u_1)+\gamma u_2,v_2\>_{\Gamma,\beta}=\<\Bb \uu, \vv\>_\mathcal{H}.
\end{align*}
This implies $\uu \in D(\Aa)$ and $\Aa\uu=\Bb\uu$.
\end{proof}

It is a consequence of Theorem \ref{main} that the semigroup $\Tt$ governs the system \eqref{1-5}--\eqref{1-9}. As the semigroup is analytic, the solution is $C^\infty$ in time so that
$(\uu(t))_{t>0} = (\Tt (t) (u_{1,0}, u_{2,0}))_{t>0}$ satisfies Equations \eqref{1-5} and \eqref{1-6} in a classical (in time) sense. Furthermore it shows that $\uu(t)\in D(\Aa^\infty)$ for $t>0$. Coming back to our initial system \eqref{1-1}--\eqref{1-4}, we immediately see that $u=u_1$ solves Equation \eqref{1-1}. The question remains in which way the Wentzell boundary condition \eqref{1-2} is satisfied. However, as $\uu(t) \in D(\Aa^\infty)$ for $t>0$, we  know in particular that $\uu(t),~ \Aa \uu(t) \in D(\a)$ for $t>0$. Hence we obtain $\tr ((\Aa \uu)_1)=(\Aa \uu)_2$ and $\tr u_1=u_2$, yielding
$\tr (\Delta(\alpha \Delta)u_1)=-\beta\partial_\nu(\alpha\Delta)u_1+\gamma u_2=-\beta\partial_\nu(\alpha\Delta)u_1+\gamma \tr u_1$.

%
%
This proves that the Wentzell boundary condition is satisfied in the sense of traces for $t>0$. Thus $u=u_1$ satisfies \eqref{1-1}--\eqref{1-4}.

\begin{bemerkung}\label{r.explain}
We point out that the system \eqref{1-1}--\eqref{1-4} has to be interpreted in such a way that $u_0$ is sufficiently smooth to have a trace on the boundary, say $u_0\in H^1(\Omega)$; in this setting, the solutions of \eqref{1-1}--\eqref{1-4} are thus in a one-to-one correspondence with the solutions of \eqref{1-5}--\eqref{1-9} with $u_{1,0}= u_0|_\Omega$ and $u_{2,0}= u_0|_{\Gamma}$. In our semigroup approach, however, $u_{2,0}$ can be chosen independently of $u_{1,0}$ and, by the above, all of these solutions are (distinct!) solutions of \eqref{1-1}--\eqref{1-4}. In a way, choosing $u_{2,0}$ different from $\tr u_{1,0}$ corresponds precisely to having some free energy on the boundary, which was a main motivation to consider Wentzell boundary conditions in the first place.
\end{bemerkung}

We now study the case of smooth domain and coefficients. For simplicity, we assume for the rest of this section that $\Omega$ is a bounded and infinitely smooth domain and that $\alpha\in C^\infty(\overline \Omega)$, $\beta,\gamma\in C^\infty(\Gamma)$ with $\alpha\ge \eta$ and $\beta\ge \eta$ on $\overline\Omega$ and $\Gamma$ for some constant $\eta>0$, respectively. In this case, it is natural to start with the strong definition of the operator. More precisely, we define the operator $\mathcal A_0$ in $\mathcal H$   by
\begin{align*}
 D(\mathcal A_0) \coloneqq \big\{\uu = (u_1,\tr u_1)\,| &\;u_1 \in C^4(\overline\Omega),\,  \tr (\Delta (\alpha \Delta) u_1)+\beta \partial_\nu (\alpha \Delta) u_1-\gamma \tr u_1=0, \\
& \partial_\nu u_1=0 \textrm{ on } \Gamma\big\} \subseteq\mathcal H
\end{align*}
and
\[ \mathcal A_0 \uu \coloneqq \begin{pmatrix}
  \Delta (\alpha\Delta) u_1\\ \tr (\Delta (\alpha\Delta) u_1)
\end{pmatrix} = \begin{pmatrix}
\Delta (\alpha\Delta ) & 0 \\ -\beta\partial_\nu (\alpha\Delta)   &  \gamma
\end{pmatrix}  \binom{u_1}{u_2}\quad \text{ for } \uu\in D(\mathcal A_0).\]

\begin{lemma}
  \label{4.1} In the smooth situation, the operator $\mathcal A_0$ is essentially self-adjoint, and its closure $\overline{\mathcal A}_0$ is given by $\mathcal A$.
\end{lemma}

\begin{proof}
The fact  that $\mathcal A_0$ is essentially self-adjoint is a special case of \cite[Theorem~1.1]{FGGR08}. As the self-adjoint extension of an essentially self-adjoint operator is unique and given by its closure (see \cite[Theorem~5.31]{Weidmann80}), we only have to show that $\mathcal A$ is an extension of $\mathcal A_0$. However, in the smooth case this is obvious from the definition of $D(\mathcal A_0)$ and the description of $D(\mathcal A)$ in Theorem~\ref{main}.
\end{proof}

We remark that even in the smooth case, we cannot expect that for $\uu\in D(\mathcal A)$ the first component $u_1$ belongs to $H^4(\Omega)$. However, we can show $u_1\in H^{7/2}(\Omega)$. To this end, we use a version of elliptic regularity which includes weighted Sobolev spaces $\Xi^s(\Omega)$, $s\in \R$, see \cite[Sections~2.6 and 2.7]{LM72}. For our application, it is enough to know that for all $s>0$, the space $\Xi^{-s}(\Omega)$ is continuously embedded into $L^2(\Omega)$. This follows by duality from the dense embedding $\Xi^s(\Omega)\subseteq L^2(\Omega)$, see \cite[Chapter 2, (6.20)--(6.21)]{LM72}.

\begin{satz}\label{HighReg}
In the smooth situation, we have
\[ D(\mathcal A)=\{\uu \in \Hh \,|\, u_1 \in H^{7/2}(\Omega), \,\Delta(\alpha\Delta) u_1 \in L^2(\Omega),\, \partial_\nu u_1=0,\,  u_2=\tr u_1 \}.\]
\end{satz}

\begin{proof}
First, let $\uu$ belong to the space on the right-hand side. From $u_1\in H^{7/2}(\Omega)$ and $\alpha\in C^\infty(\overline\Omega)$, we obtain $u_1\in H^{3/2}_\Delta(\Omega)$ and $ \alpha\Delta u_1\in H^{3/2}(\Omega)$.  Now $\Delta(\alpha\Delta) u_1 \in L^2(\Omega)$ yields $\alpha\Delta u_1 \in H^{3/2}_\Delta(\Omega)$, and with the description of $D(\mathcal A)$ in Theorem~\ref{main} we see that $\uu\in D(\mathcal A)$.

For the other direction, let $\uu\in D(\mathcal A)$. We apply  the elliptic regularity result from \cite[Rem.\ 2.7.2]{LM72}, setting there  $A=\Delta(\alpha\Delta)+I$, $B_0=\partial_\nu$, $B_1=-\beta \partial_\nu (\alpha \Delta)+\gamma \tr $, and $s=\frac{7}{2}$. We obtain
\begin{align*}
\|u_1\|_{H^{7/2}(\Omega)}&\leq C\Big(\|u_1+\Delta(\alpha\Delta u_1)\|_{\Xi^{-1/2}(\Omega)}+\|\partial_\nu u_1\|_{H^2(\Gamma)}+\|\gamma u_2-\beta\partial_\nu (\alpha \Delta u_1)\|_{\Gamma}\Big)\\
&\leq C\Big(\|u_1\|_\Omega+\|\Delta(\alpha\Delta u_1)\|_\Omega+\|\gamma u_2-\beta\partial_\nu (\alpha \Delta u_1)\|_{\Gamma} \Big)\\
&\le C \Big( \|\mathcal A\uu\|_{\Hh} + \|\uu\|_{\Hh}\Big).
\end{align*}
Here, for the second inequality, we use the continuous embedding $L^2(\Omega)\subseteq \Xi^{-1/2}(\Omega)$ (see above) and the fact that $\partial_\nu u_1=0$. We see that  $u_1\in H^{7/2}(\Omega)$ and  that $D(\mathcal A)$ is continuously embedded into the space on the  right-hand side.
\end{proof}

\begin{bemerkung}
We assumed the domain and the coefficients to be infinitely smooth, as the theory from \cite{LM72} is formulated in this setting. However, the proofs  are based on elliptic regularity up to order $4$, duality and interpolation, which shows that it is, e.g., sufficient to assume  $\Omega$ to have  a $C^4$-boundary as well as $\alpha \in C^4(\overline\Omega)$, $\beta,\gamma \in C^{3+\eps}(\Gamma)$. This regularity  was considered in \cite{FGGR08}, and thus Theorem~\ref{HighReg} gives the precise domain of the self-adjoint extension of the operator $\Aa_0$. However, we omit the formal  proof and  technical details for this.
\end{bemerkung}

\begin{bemerkung}
We would like to point out that in the rough case there are examples for domains where we can find $\uu =(u_1, u_2) \in D(\Aa)$ such that $u_1 \not\in H^{3/2+\eps}(\Omega)$ for any $\eps>0$. This behaviour is known for the Neumann Laplacian. For $d=2$, there are even $C^1$-domains $\Omega$ and functions $u \in D(\Delta_N)$ such that $\Delta u=f \in C^\infty(\overline\Omega)$, $\partial_\nu u=0$ and $u \not\in H^{3/2+\eps}(\Omega)$ (cf.\ \cite[Section 3]{Cos19}). If we take $\alpha\equiv 1$, it follows from Theorem \ref{main} that for any such example $u$ we have $(u,\tr u) \in D(\Aa)$, as $f=\Delta u \in C^\infty(\overline\Omega)\subseteq H^{3/2}_\Delta(\Omega)$.
This shows that in the Lipschitz setting, one cannot expect more regularity than
$H^{3/2}(\Omega)$ for functions belonging to $D(\Aa)$, in contrast to the smooth setting, where Theorem~\ref{HighReg}
yields the regularity $H^{7/2}(\Omega)$.

This significant difference in regularity between the rough and the smooth setting also suggests that there is little hope in tackling  Lipschitz domains by approximating them with smooth domains. That domain approximation is a delicate business for higher order elliptic operators subject to boundary conditions is a well-known phenomenon. This is illustrated by the \emph{Babu\v{s}ka paradox}, where a circular domain is approximated by a sequence of polygons but the solutions do not converge to the solution on the smooth domain (see, e.g., \cite[Section~2.2]{Swe09} for details).
\end{bemerkung}

\section{H\"older Continuity of the solution}\label{classhoel}

As a preparation to prove H\"older regularity in Theorem \ref{HoelReg}, we establish some results concerning weak solutions of the inhomogeneous Neumann problem
\begin{equation}\label{inhom}
\begin{alignedat}{4}
\Delta u& =f&\;&\text{ in } \Omega, \\
\partial_\nu u & =g&& \text{ on }\Gamma.
\end{alignedat}
\end{equation}
By a \emph{weak solution} of \eqref{inhom}, we mean a function $u\in H^1(\Omega)$ such that
\[
-\la \nabla u, \nabla v\ra_\Omega = \la f, v\ra_\Omega + \la g, \tr v\ra_\Gamma
\]
for all $v\in H^1(\Omega)$.
Naturally, the data $f$ and $g$ have to have enough integrability such that these integrals are well defined.  Note that, as a consequence of Proposition \ref{LemGreen}, a weak solution of \eqref{inhom} automatically belongs to the space $H^{3/2}_\Delta(\Omega)$.

In what follows, we write $\|f\|_{\Omega, p}$ for the norm of $f$ in $L^p(\Omega)$ and $\|g\|_{\Gamma, p}$ for the norm of $g$ in $L^p(\Gamma)$.
We begin by recalling the following result from \cite{Nit10}, in which $C^\alpha(\Omega)$ refers to the space of $\alpha$-H\"older continuous functions on $\Omega$. Note that every function $u\in C^\alpha(\Omega)$ can be extended uniquely
to a (H\"older) continuous function on $\overline{\Omega}$.

\begin{lemma}\label{Hoel}

Assume that $d\geq 2$, $f\in L^{\frac{d}{2}+\eps}(\Omega)$, and $g\in L^{d-1+\eps}(\Gamma)$ (or $d=1$,
$f\in L^1(\Omega)$, and $g\in L^1(\Gamma)$).  Then, there exists $\alpha \in (0,1)$ such that if $u \in H^1(\Omega)$ is a weak solution of \eqref{inhom}, then $u \in C^{\alpha}(\Omega)$ and
\[
\|u\|_{C^{\alpha}(\Omega)}\leq C \big(\|u\|_{\Omega, 2}+\|f\|_{\Omega,\frac{d}{2}+\eps}+\|g\|_{\Gamma, d-1+\eps}\big).
\]
\end{lemma}
\begin{proof}
This is \cite[Theorem 3.1.6]{Nit10}. Note that we are in the situation of \cite[Remark 3.1.7]{Nit10}.
\end{proof}

 Lemma \ref{Hoel} allows us in particular to estimate $\|u\|_{\Omega, \infty}$ and $\|\tr u\|_{\Gamma, \infty}$
for solutions of \eqref{inhom}, provided the data have high enough integrability. We prove next that solutions $u\in H^{3/2}_\Delta (\Omega)$ of \eqref{inhom} have higher integrability than the data.

\begin{lemma}\label{IPolr}
  Let $d\ge 2,\, p\in (2,\infty)$. Then there is a constant $C_0>0$ such that whenever $u \in H^{3/2}_\Delta(\Omega)$ is a weak solution of \eqref{inhom} with $(f,g)\in L^{p}(\Omega)\times L^{p}(\Gamma)$, then $(u,\tr u)\in L^{\phi(p)}(\Omega)\times L^{\phi(p)}(\Gamma)$ and
\[ \|u\|_{\Omega, \phi(p)} + \|\tr u\|_{\Gamma, \phi(p)} \le C_0 \Big( \|u\|_{\Omega, 2} + \|f\|_{\Omega, p} + \|g\|_{\Gamma,p}\Big), \]
where
\[\phi(p) \coloneqq \begin{cases}
  \frac{d-2}{d-p}\, p & \text{ if } p\in (2,d),\\
  \infty & \text{ if } p\in [d,\infty).
\end{cases}\]
\end{lemma}

\begin{proof}
We first consider the end-point cases $p=2$ and $p=\infty$, then use interpolation.

As for $p=\infty$, note that for small enough $\eps$, we have $d/2+\eps, d-1+\eps \leq d$, so that
Lemma \ref{Hoel} yields
\begin{equation}\label{3-1}
\|u\|_{\Omega, \infty} + \|\tr u\|_{\Gamma, \infty}  \leq \|u\|_{C^{\alpha}(\Omega)} \leq
C\big( \|u\|_{\Omega, 2} +\|f\|_{\Omega, d} + \|g\|_{\Gamma, d}\big).
\end{equation}

For $p=2$, we use the continuity of the trace operator from $H^1(\Omega)$ to $L^2(\Gamma)$ and obtain with Cauchy--Schwarz's and Young's inequality
\begin{align*}
\|u\|^2_{\Omega, 2}+\|\tr u\|^2_{\Gamma, 2} &\leq C \Big( \|u\|_{\Omega, 2}^2+ \|\nabla u\|_{\Omega, 2}^2\Big)\\
& = C\Big( \|u\|_{\Omega, 2}^2+\<-\Delta u, u\>_\Omega + \<\partial_\nu u, \tr u\>_\Gamma \Big) \\
&\leq C\Big( \|u\|_{\Omega, 2}^2+ \|\Delta u\|_{\Omega, 2}^2+ \|\partial_\nu u\|_{\Gamma,2}^2\Big) +\tfrac12 \|\tr u\|_{\Gamma, 2}^2.
\end{align*}
This yields
\begin{equation}\label{3-2}
\|u\|_{\Omega, 2}+\|\tr u\|_{\Gamma, 2} \leq C \Big( \|u\|_{\Omega, 2}+\|f\|_{\Omega, 2}+\|g\|_{\Gamma, 2} \Big).
\end{equation}

In order to interpolate between \eqref{3-1} and \eqref{3-2}, let us first prove that if $u$ is a solution of \eqref{inhom} with data $f$ and $g$, then
\begin{equation}
  \label{3-3}
  \|u\|_{H^{3/2}_\Delta(\Omega)} \le C \Big( \|u\|_{\Omega, 2} + \|f\|_{\Omega, 2} + \|g\|_{\Gamma, 2}\Big).
\end{equation}
As  the map $\tau_N =\partial_\nu\colon H^{3/2}_\Delta(\Omega)\to L^2(\Gamma)$ has a bounded right-inverse
$e_N\colon L^2(\Gamma)\to H^{3/2}_\Delta(\Omega)$,
we can  set $v:= e_N g\in H^{3/2}_\Delta(\Omega)$. Then $w := u - v$ is a solution of
\begin{alignat*}{4}
  \Delta w & = f - \Delta v &\quad& \text{ in } \Omega,\\
  \partial_\nu w & = 0&& \text{ on }\Gamma.
\end{alignat*}
In particular, $w\in D(\Delta_N)$ and therefore (see \cite[Corollary 8.7.4]{BHS20})
\begin{align*}
   \|w\|_{H^{3/2}_\Delta(\Omega)} & \le C \big( \|w\|_{\Omega, 2} + \|f\|_{\Omega, 2} + \|\Delta v\|_{\Omega, 2} \big)\\
   & \le C \big( \|u\|_{\Omega, 2} + \|v\|_{H^{3/2}_\Delta(\Omega)} + \|f\|_{\Omega, 2} \big)\\
   & \le C \big( \|u\|_{\Omega, 2} +  \|f\|_{\Omega, 2} + \|g\|_{\Gamma, 2}\big).
\end{align*}
In the last step, we used the continuity of $e_N$. Thus,
\[ \|u \|_{H^{3/2}_\Delta(\Omega)} \le \|v\|_{H^{3/2}_\Delta(\Omega)}+ \|w\|_{H^{3/2}_\Delta(\Omega)} \le  C \big( \|u\|_{\Omega, 2} +  \|f\|_{\Omega, 2} + \|g\|_{\Omega, 2}\big),\]
which shows \eqref{3-3}.

For the interpolation, let $X_0 := F(H^{3/2}_\Delta(\Omega))$, where
\[ F\colon H^{3/2}_\Delta(\Omega)\to L^2(\Omega)\times L^2(\Omega)\times L^2(\Gamma),\;
u\mapsto (u,\Delta u, \partial_\nu u).\]
Then $F\colon H^{3/2}_\Delta(\Omega)\to X_0$ is bounded, bijective, and its inverse is bounded due to \eqref{3-3}. So $F$ is an isomorphism of normed spaces, and, as $H^{3/2}_\Delta(\Omega)$ is a Banach space, the same is true for $X_0$. Let $Z_1:= L^2(\Omega)\times L^{d}(\Omega)\times L^{d}(\Gamma)$ and $X_1 := X_0\cap Z_1$. By \eqref{3-2}, the linear operator
\[ T\colon X_0\to Y_0 := L^2(\Omega)\times L^2(\Gamma),\; (u,\Delta u, \partial_\nu u)  \mapsto
(u,\tr u) \]
is well-defined and bounded. By \eqref{3-1}, the same holds for its restriction
\[ T\colon X_1 \to Y_1 := L^\infty(\Omega)\times L^\infty(\Gamma). \]
Complex interpolation shows that $T\colon [X_0,X_1]_\theta \to [Y_0,Y_1]_\theta$ is continuous for all
$\theta\in (0,1)$. To identify the interpolation spaces, recall from \cite[Theorem 1.18.1]{Tri95} that
complex interpolation of tuples of $L^p$-spaces yields the tuple of interpolated spaces in the sense of
\[ [L^{p_0}(\Omega)\times L^{q_0}(\Gamma), L^{p_1}(\Omega)\times L^{q_1}(\Gamma)]_\theta = [L^{p_0}(\Omega),L^{p_1}(\Omega)]_\theta \times
[L^{q_0}(\Gamma)\times L^{q_1}(\Gamma)]_\theta \]
for all $p_0,p_1,q_0,q_1\in [1,\infty]$.
Moreover, we have the equality $[L^{p_0}(\Omega), L^{p_1}(\Omega)]_\theta = L^p(\Omega)$ (and a similar equality for $\Gamma$) for $\frac 1p=\frac{1-\theta}{p_0} + \frac\theta{p_1}$ in the sense of equivalent norms, see \cite[Theorem 1.18.6/2]{Tri95}. From this, we obtain for all
$\theta\in (0,1)$ the
continuity of $T\colon X_0\cap Z_\theta \to Y_\theta$, where
\[ Z_\theta := L^2(\Omega)\times L^{p}(\Omega) \times L^{p}(\Gamma) \]
and $Y_\theta := L^{\phi(p)}(\Omega)\times L^{\phi(p)}(\Gamma)$ with  $p$ and $\phi(p)$ being
defined by $\frac 1p = \frac{1-\theta}2+ \frac \theta d$ and
$\frac 1{\phi(p)}=\frac{1-\theta}2$. For $p\in (2,d)$, the first equality yields
$\theta = \frac{d(p-2)}{(d-2)p}$, and the second equality gives
\[ \phi(p) = \frac{2}{1-\theta} = \frac{d-2}{d-p}\; p .\]
Now the continuity of $T\colon X_0\cap Z_\theta \to Y_\theta$ shows that for all
$u\in H^{3/2}_\Delta(\Omega)$ we have
\[ \|u\|_{\Omega, \phi(p)} + \|\tr u\|_{\Gamma, \phi(p)} \le C \Big( \|u\|_{\Omega, 2} + \|\Delta u\|_{\Omega, p} + \|\partial_\nu u\|_{\Gamma, p}\Big), \]
which  proves the lemma for $p\in (2,d)$. For $p\ge d$ the statement follows
directly from \eqref{3-1}.
\end{proof}

We obtain the following corollary about the integrability of elements of $D(\Aa)$.

\begin{korollar}\label{boot}
Let $r> 2$.
If $\uu \in D(\Aa) \cap (L^r(\Omega)\times L^r(\Gamma))$ and $\Aa\uu \in L^r(\Omega)\times L^r(\Gamma)$, then
$\uu \in L^{\phi^2(r)}(\Omega)\times L^{\phi^2(r)}(\Gamma)$ and $\Delta u_1 \in L^{\phi(r)}(\Omega)$.
\end{korollar}

\begin{proof}
By Theorem~\ref{main}, we have for $\uu \in D(\Aa)$
\begin{align*}
(\Aa\uu)_1&=\Delta (\alpha \Delta) u_1,\\
 (\Aa\uu)_2&=-\beta \partial_\nu(\alpha \Delta) u_1+\gamma u_2.
\end{align*}
Thus, if $\uu$ satisfies the assumption of this corollary, then $\alpha \Delta u_1$
solves the inhomogeneous Neumann problem
\begin{align*}
\Delta (\alpha \Delta) u_1&=(\Aa\uu)_1 \in L^r(\Omega) \\
  \partial_\nu(\alpha \Delta) u_1&=-\beta^{-1}(\Aa\uu)_2+\beta^{-1}\gamma u_2 \in L^r(\Gamma).
\end{align*}
By Lemma \ref{IPolr}, $\alpha \Delta u_1 \in L^{\phi(r)}$, yielding $\Delta u_1 \in L^{\phi(r)}(\Omega)$ as well.
Since $\uu \in D(\Aa)$, we also know that $\partial_\nu u_1=0$ and $u_2=\tr u_1$, so that
$u_1$ solves the homogeneous Neumann problem
\begin{align*}
\Delta  u_1&=\Delta u_1 \in L^{\phi(r)}(\Omega) \\
  \partial_\nu u_1&=0 \in L^{\phi(r)}(\Gamma).
\end{align*}
Applying Lemma \ref{IPolr} once more, we obtain $u_1 \in L^{\phi^2(r)}(\Gamma)$ and $u_2=\tr u_1 \in L^{\phi^2(r)}(\Gamma)$ as claimed.
\end{proof}

We can now prove the main result of this section.

\begin{satz}\label{HoelReg}
Let $\uu \in D(\Aa^\infty)$. Then $u_1 \in C^{\alpha}(\Omega)$ for some $\alpha \in (0,1)$.
\end{satz}

\begin{proof}
Let $\uu \in D(\Aa)$. Then $\Delta u_1 \in H^{3/2}_\Delta(\Omega) \subseteq H^1(\Omega)$ and $\partial_\nu u_1 =0 \in L^\infty(\Gamma)$. If $d\leq 5$, then, by Sobolev embedding (see \cite[Theorem 4.12]{AF03}), $\Delta u_1 \in L^{\frac{d}{2}+\eps}(\Omega)$ and Lemma \ref{Hoel}  yields $u_1 \in C^{\alpha}(\Omega)$.

Now consider the case $d \geq 6$. In this case the Sobolev embedding yields $\Delta u_1 \in L^\frac{2d}{d-2}(\Omega)$.
Setting $r_1 \coloneqq \phi(\frac{2d}{d-2})>2$, Lemma \ref{IPolr} implies
\[
\uu=(u_1, \tr u_1) \in L^{r_1}(\Omega)\times  L^{r_1}(\Gamma).
\]
Thus, $D(\Aa) \subseteq  L^{r_1}(\Omega)\times  L^{r_1}(\Gamma)$. Inductively, we obtain
$D(\Aa^k)\subseteq L^{r_k}(\Omega)\times  L^{r_k}(\Gamma)$, where $r_k=\phi^2(r_{k-1})=\phi^{2k-1}(\frac{2d}{d-2})$.
Indeed, assume this statement is true for some $k$ and let $\uu\in D(\Aa^{k+1})$. Then $\uu \in D(\Aa^k) \subseteq D(\Aa)$
and $\Aa\uu \in D(\Aa^k)$. By induction hypothesis, $\uu, \Aa\uu \in   L^{r_{k}}(\Omega)\times  L^{r_{k}}(\Gamma)$, and Corollary \ref{boot} yields $\uu \in L^{\phi^2(r_{k})}(\Omega)\times  L^{\phi^2(r_{k})}(\Gamma)=L^{r_{k+1}}(\Omega)\times  L^{r_{k+1}}(\Gamma)$ as well as $\Delta u_1 \in L^{\phi(r_{k})}(\Omega)$.

From the structure of the map $\varphi$ it is clear that $(r_k)_{k\in\mathbb N}$ is an increasing sequence that tends to $\infty$.
We thus find $k_0\in \mathbb{N}$ such that $D(\Aa^{k_0-1}) \subseteq L^d(\Omega)\times L^d(\Gamma)$. For $\uu \in D(\Aa^{k_0})$,
we have $\Delta u_1 \in L^{\varphi (d)}(\Omega)$ and $\partial_\nu u_1 \in L^\infty(\Gamma)$. Thus, Lemma \ref{Hoel}  implies $u_1 \in C^{\alpha}(\Omega)$ as claimed.
\end{proof}

\begin{bemerkung}\label{CorHoel}
The proof of Theorem \ref{HoelReg} actually shows that given the dimension $d$, there exists a number $k_0\in \mathbb{N}$, depending only on $d$, such that for $\uu \in D(\Aa^{k_0})$ we have $u_1\in C^{\alpha}(\Omega)$.
\end{bemerkung}

\section{Spectral decomposition and asymptotic behavior}

In this section, we prove that we can find an orthonormal basis of $\Hh$ consisting of eigenfunctions of $\Aa$ and study the long-time behavior of the semigroup $\Tt$. We begin with the following lemma.

\begin{lemma}
The operator $\Aa$  has compact resolvent.
\end{lemma}

\begin{proof}
We have to show that the embedding $D(\Aa)\subseteq\Hh$ is compact. By Theorem~\ref{main}, we know that the operator $\pi_1\colon D(\Aa)\to H^{3/2}(\Omega),\, \uu\mapsto u_1$ is well defined. We show that $\pi_1$ is closed. For this, let $\uu_n = (u_1^n,u_2^n), \,n\in\mathbb N,$ be a sequence in $D(\Aa)$ with $\uu_n\to \uu_0=(u_1^0,u_2^0)$ in $D(\Aa)$
and $\pi_1 \uu_n \to v_1$ in $H^{3/2}(\Omega)$. Then $u_1^{n} \to u_1^{0}$ in $L^2(\Omega)$ and also $u_1^{n} \to v_1$ in $L^2(\Omega)$, which shows $v_1 = u_1^{0} = \pi_1 \uu_0$. Thus $\pi_1$ is closed and, by the closed graph theorem, bounded.

Let $(\uu_n)_{n\in\mathbb N}$ be a bounded sequence in $D(\Aa)$. As $\pi_1$ is bounded, the sequence $(u_1^{n})_{n\in\mathbb N}$ is bounded in $H^{3/2}(\Omega)$ and therefore also in $H^1(\Omega)$. By the theorem of Rellich--Kondrachov (see \cite[Theorem~6.3]{AF03}), there exists a subsequence which converges in $L^2(\Omega)$.
As $\tr\colon H^1(\Omega)\to H^{1/2}(\Gamma)$ is continuous and $H^{1/2}(\Gamma)$ is compactly embedded into $L^2(\Gamma)$ (see \cite[Equation (2.17)]{GM11}), we have convergence of another subsequence of $(\tr u_1^{n})_{n\in\mathbb N}$ in $L^2(\Gamma)$. From this and $\tr u_1^{n} = u_2^{n}$, we see  that there exists a subsequence of $(\uu_n)_{n\in\mathbb N}$ which converges in $\Hh$. This shows the compactness of  the embedding  $D(\Aa)\subseteq \Hh$.
\end{proof}

We now obtain the following spectral decomposition of our operator $\Aa$.

\begin{korollar}\label{c.spectrum}
There exists an orthonormal basis $(\ee_n)_{n\in\mathbb{N}}$ of $\Hh$ consisting of eigenfunctions of $\Aa$, say
$\Aa \ee_n = \lambda_n \ee_n$, where the sequence $\lambda_n$ is increasing to $\infty$.
Moreover, as $\ee_n \in D(\Aa^\infty)$, it has a H\"older continuous representative in the sense that
there exists a function $e_n \in C^{\alpha}(\Omega)$ such that $\ee_n = (e_n|_{\Omega}, e_n|_{\Gamma})$.
Finally, the semigroup $\Tt$ can be represented as
\begin{align} \label{sgdecomp}
\Tt (t)\ff=\sum_{k=1}^\infty e^{-\lambda_k t}\<\ff,\ee_k\>_{\Hh}\ee_k.
\end{align}
\end{korollar}

From the representation \eqref{sgdecomp} we can obtain information about the asymptotic behavior in a standard way. For this, however, we need some additional information about the first eigenvalue,  wich we obtain by making use of the following facts.

\begin{bemerkung}\label{remmay}
The first eigenvalue $\lambda_1$ of $\Aa$ can be obtained by minimizing the \emph{Rayleigh quotient}:
\[
\lambda_1 = \inf_{\uu\in D(\a)\setminus \{0\}} \frac{\a(\uu)}{\|\uu\|^2_{\Hh}}.
\]
Moreover, the infimum is in fact a minimum and every minimizer is an eigenfunction for $\lambda_1$.
Thus,
\[
\lambda_1 = \inf_{\uu\in D(\a)\setminus \{0\}} \frac{\a(\uu)}{\|\uu\|^2_{\Hh}}
= \inf_{\uu\in D(\Aa)\setminus \{0\}} \frac{\la \Aa \uu, \uu \ra_\Hh}{\|\uu\|^2_{\Hh}}.
\]
\end{bemerkung}

\begin{lemma}\label{kernel}~

\begin{enumerate}[(i)]
\item If $\gamma= 0$ almost everywhere, then $\lambda_1 = 0$ and $\ker(\Aa)=\mathrm{span}\{(\one_{\Omega}, \one_\Gamma)\}$.
\item If $\gamma\geq 0$ and $\gamma>0$ on a set of positive measure, then $\lambda_1 >0$ and we have $\ker(\Aa)=\{0\}$.
\item If $\int_\Gamma \gamma\, \dx S <0$, then $\lambda_1<0$.
\end{enumerate}
\end{lemma}

\begin{proof}
In cases (i) and (ii), $\a$ is accretive, so we have $\lambda_1 \geq 0$. Thus, whether $\lambda_1=0$ or $\lambda_1>0$ depends only on $\ker (\Aa)$.\smallskip

(i) Suppose $\gamma= 0$ almost everywhere. Then any constant function belongs to the kernel of $\Aa$ and hence $\lambda_1(\Aa)=0$. Let us prove that any element of $\ker (\Aa)$ is necessarily constant.
To that end, let $\uu \in \mathrm{ker}(\Aa)\subseteq D(\Aa)\subseteq D(\a)$. Then
\[
0=\<\Aa\uu,\uu\>_{\mathscr{H}}=\a(\uu,\uu)=\int_\Omega \alpha |\Delta u_1|^2 \mathrm{d}x.
\]
It follows that $\alpha |\Delta u_1|^2=0$ and hence, since $\alpha (x)\geq \eta$, $\Delta u_1 =0$. As, moreover,
$\partial_\nu u_1=0$, we have $u_1 \in \ker(\Delta_N)$. But only constants lie in the kernel of the Neumann Laplacian. Indeed, the Neumann Laplacian is associated to the form $\a_N(u,v)=\<\nabla u, \nabla v\>_\Omega$ defined on $H^1(\Omega)$.
Arguing as above we find for $u \in \ker \Delta_N$ that $\|\nabla u\|_\Omega^2=0$ and thus $\nabla u=0$ so $u$ is a constant. It follows that $u_1$ (hence also $u_2=\tr u_1)$ is constant almost everywhere.\smallskip

(ii) Now let $\gamma\ge 0$, $\gamma\not=0$ and $ \uu \in \ker (\Aa)$. As above we see that
\[
0= \<\Aa\uu,\uu\>_{\mathscr{H}}=\a(\uu,\uu) =\int_\Omega \alpha |\Delta u_1|^2 \mathrm{d}x+\int_\Gamma \gamma |u_2|^2 \dx S.
\]
But then each of these integrals has to be zero. Arguing as above shows that $u_1\in \ker (\Delta_N)$ and hence $u_1\equiv c$ for some constant. But then $u_2 =\tr u_1 \equiv c$. As $\gamma \not=0$, we find some set $P\subseteq \Gamma$ of positive measure and $\eps>0$ such that $\gamma (x) \geq \eps$ for every $x\in P$. This implies
\[
0\geq \int_\Gamma \gamma c^2\, dS \geq \eps c^2 |P|,
\]
which, in turn, implies $c=0$.
\smallskip

(iii) Plugging $(\one_\Omega,\one_\Gamma)\in D(\a)$ into the Rayleigh quotient, we obtain a negative value as
$\int_\Gamma \gamma \, \dx S < 0$. Thus $\lambda_1 < 0$.
\end{proof}

We can now characterize the asymptotic behavior of our semigroup.

\begin{satz}\

\begin{enumerate}[(i)]
\item If $\gamma = 0$ almost everywhere, then $\|\Tt(t)\ff - \bar\ff\|_\Hh\leq e^{-\lambda_2 t}\|\ff\|_{\Hh}$ for all $\ff\in\Hh$, where
\[
\bar \ff := \frac{1}{|\Omega|+|\Gamma|}\left(\int_\Omega f_1  \mathrm{d}x+\int_\Gamma f_2  \mathrm{d}S\right)(\one_\Omega,\one_\Gamma),
\]
and $\lambda_2>0$ is the second eigenvalue of $\Aa$.
\item If $\gamma\geq 0$ and $\gamma>0$ on a set of positive measure, then $\|\Tt(t)\ff \|_\Hh\leq e^{-\lambda_1 t}\|\ff\|_{\Hh}$ holds for all $\ff\in\Hh$. Thus, in this case, the semigroup $\Tt$ is exponentially stable.
\item If $\int_\Gamma \gamma \, \dx S < 0$, then $\|\Tt(t)\| = e^{-\lambda_1 t} \to \infty$ as $t\to \infty$.
\end{enumerate}
\end{satz}

\begin{proof}
As for (i), observe that in this case $\bar \ff = e^{-\lambda_1 t} \la \ff, \ee_1\ra_{\Hh}\ee_1$ in view of Lemma \ref{kernel}. Thus
\eqref{sgdecomp} and Parseval's identity yield
\[
\|\Tt(t)\ff - \bar \ff\|_\Hh^2= \sum_{k=2}^\infty e^{-\lambda_k t}|\la \ff, \ee_k\ra_{\Hh}|^2 \leq
\sum_{k=2}^\infty e^{-\lambda_2 t}|\la \ff, \ee_k\ra_{\Hh}|^2 \leq e^{-\lambda_2 t}\|f\|^2_{\Hh}.
\]
This proves (i). In case (ii) we have $\lambda_1>0$ (see again Lemma \ref{kernel}), and (ii) follows by a similar computation.
(iii) follows by considering an eigenvalue corresponding to the eigenvalue $\lambda_1$.
\end{proof}

\section{Eventual Positivity}

We have seen in Proposition \ref{nonpos} that the semigroup associated to the operator $\Aa$ is never positive. This is hardly surprising, as this is the expected behavior of semigroups generated by the Bi-Laplacian subject to `classical' boundary conditions. However, for some of these boundary conditions, like `sliding'  boundary conditions or Dirichlet (in this context also called `clamped') boundary conditions on certain domains, the semigroup is, in a sense, eventually positive.
As this behavior is also observed for other operators (including the Dirichlet-to-Neumann operator), recently a systematic treatment of this phenomenon was initiated, see \cite{DGK16b, DGK16a, DG18}.

In this section we will prove that in the case $\gamma = 0$, the semigroup $\Tt$ is \emph{eventually positive} in the sense that there is some $t_0>0$ such that for every $\ff \in \Hh$ with $\ff \geq 0$ but $\ff\neq 0$ there exists an $\eps>0$ such that $\Tt (t)\ff(x) \geq \eps$ for all $t\geq t_0$ and (considering Theorem \ref{HoelReg}) all
$x\in \Omega\cup\Gamma$; in the language of \cite{DG18} it would be more precise to call this behavior \emph{uniform, eventual strong positivity with respect to the quasi-interior point $\one$}. The  term `uniform' refers to the fact that the time $t_0$ can be chosen independently of the function $\ff$. In our situation this uniformity follows from the self-adjointness of $\Aa$ (cf.\ \cite[Cor.\ 3.5]{DG18}).

The case where $\gamma\geq 0$ but $\gamma \neq 0$ is more involved. In this case the function
$\one$ does not satisfy the boundary condition and we have to replace it with some other quasi-interior point, i.e.\ a strictly positive function. In  practice, if the first eigenfunction of the generator of the semigroup is positive, one uses this function. In fact, for a semigroup to be (even individually) eventually strongly positive, it is also necessary that the first eigenfunction is positive. However, for the Bi-Laplacian with Dirichlet (or clamped) boundary conditions it is known that for some domains (see \cite{Swe01} for a survey) the first eigenfunction changes sign.

As it turns out, Dirichlet boundary conditions appear as a limiting case of our general boundary conditions. At the end of this section, we will prove that we can deduce from this that also in our situation, it can happen that the first eigenfunction of our operator $\Aa$ changes sign so that the semigroup $\Tt$ is \emph{not} eventually positive in any sense in this situation.

But let us start with $\gamma=0$.

\begin{satz}\label{EventPos}
Let $\gamma=0$. Then the semigroup $\Tt$ is eventually positive in the sense defined above.
\end{satz}

\begin{proof}
We apply \cite[Cor.\ 3.5]{DG18} for the quasi-interior point $\one$ of $\Hh$. Note that, as a consequence of Lemma \ref{kernel},  we have $\lambda_1=0$ as $\gamma=0$ and the corresponding eigenspace is spanned by $\one$ (thus condition (iii) of \cite[Cor.\ 3.5]{DG18} is satisfied). It remains to check the other hypotheses of \cite[Cor.\ 3.5]{DG18}.
We first note that $\Tt$ is real as a consequence of Lemma \ref{nonpos}. Furthermore the operator $\Aa$ is self-adjoint due to the symmetry of the form (see Theorem \ref{t.Agenerator}). All that is left to show is that $D(\Aa^\infty)$ embeds into the ideal generated by $\one$, i.e.\ $L^\infty(\Omega)\times L^\infty(\Gamma)$. But this follows from Theorem \ref{HoelReg}.
\end{proof}

We now turn to the situation where $\gamma >0$. Let us first explain how the Bi-Laplacian with Dirichlet boundary conditions can be obtained as a limiting case. To that end, we consider a  sequence $(\gamma_n)_{n\in\mathbb N}$ in $L^\infty(\Gamma;\R)$ with $0\leq \gamma_n \leq \gamma_{n+1}$. We assume that there exists a sequence $(g_n) \subseteq (0,\infty)$ with $\gamma_n(x) \geq g_n$ for almost all $x\in \Gamma$ and such that $g_n \nearrow \infty$. We now consider the sequence $\a_n$, defined by $D(\a_n) := D(\a)$ and
\[
\a_n(\uu, \vv) \coloneqq \la \Delta u_1, \Delta v_1\ra_\Omega + \la \gamma_n u_2, v_2\ra_\Gamma.
\]
Note that we have chosen $\alpha\equiv 1$ and $\beta\equiv 1$ here. Obviously, the sequence $\a_n$ is increasing, in the sense that $D(\a_{n+1}) \subseteq D(\a_n)$ and $\a_n(\uu)\leq \a_{n+1}(\uu)$ for all $n\in \mathbb{N}$ and $\uu\in D(\a_{n+1})$. We are thus in the situation of Barry Simon's monotone convergence theorem, see \cite{Sim78}. The limiting form $\a_\infty$ is defined by setting $\a_\infty(\uu) \coloneqq \sup_{n\in \mathbb{N}} \a_n(\uu)$ for
\[
\uu \in D(\a_\infty):= \Big\{\bigcap_{n\in \mathbb{N}} D(\a_n) \, \Big|\, \sup_{n\in \mathbb{N}} \a_n(\uu) < \infty\Big\}.
\]
In our concrete situation, it is easy to see that the limiting form is given by $\a_\infty(\uu, \vv) = \la \Delta u_1, \Delta v_1\ra_\Omega$, defined on the domain
\[
D(\a_\infty) = \{ \uu \in \Hh \, |\, u_1 \in D(\Delta_N), u_2 = \tr u_1 = 0 \}.
\]
We point out that the limiting form is \emph{not} densely defined (as $\overline{D(\a_\infty)} = L^2(\Omega) \times \{0\}$).
Nevertheless, we obtain degenerate convergence of the associated operators in the strong resolvent sense (see Section 4 of \cite{Sim78}); here, for the limiting form, we have to consider the resolvent of the associated operator on
$\overline{D(\a_\infty)}$ and then extend this to $\Hh$ by setting it to 0 on $\overline{D(\a_\infty)}^\perp$.

Let us identify the operator associated to the limiting form on $\overline{D(\a_\infty)}\simeq L^2(\Omega)$. We put
$\tilde\a_\infty(u,v) = \la \Delta u, \Delta v\ra_\Omega$ for $$u, v\in D(\tilde \a_\infty) \coloneqq \{ u\in D(\Delta_N)\,|\,
\Delta^2u \in L^2(\Omega), \tr u = 0\}.$$ As a consequence of Simon's monotone convergence theorem, the form
$\a_\infty$ (thus also $\tilde \a_\infty$) is closed.

The following Lemma shows that the limiting operator is the Bi-Laplacian subject to Dirichlet boundary conditions $\tr u = 0$ and $\partial_\nu u  =0$ with maximal domain.

\begin{lemma}
The associated operator to $\tilde\a_\infty$ is given by $A_\infty=\Delta^2$ on
\[ D(A_\infty)=\{u\in D(\Delta_N) \,|\, \Delta^2 u \in L^2(\Omega), \tr u=0\}.
\]
\end{lemma}

\begin{proof}
Denote, for the moment, the operator associated to $\tilde \a_\infty$ by $A$, i.e.\ $u\in D(A)$ and $Au=f$
if and only if $u\in D(\tilde \a_\infty)$ and $\tilde \a_\infty(u,v) = \la f,v\ra_\Omega$ for all
$v\in D(\tilde\a_\infty)$. Thus, if $u\in D(A)$, by considering $v\in C_c^\infty(\Omega)$ it immediately follows that $Au=\Delta^2u = f \in L^2(\Omega)$. Since $\tr u = 0$ for all $u\in D(\tilde \a_\infty)$, we have $u\in D(A_\infty)$.

For the converse inclusion, let $u\in D(A_\infty)\subseteq D(\tilde\a_\infty)$ and put $f\coloneqq \Delta^2u=A_\infty u$.
Then $\Delta u \in H^0_\Delta (\Omega)$
Thus for all $v\in D(\tilde\a_\infty) \subseteq D(\Delta_N)$ we obtain from Lemma \ref{ExTrace}
\[
\la f, u\ra_\Omega = \la \Delta (\Delta u), v\ra_\Omega = \la \Delta u, \Delta v\ra +\la \tilde\tau_N \Delta u, \tr v\ra_{\mathcal{G}_0'\times\mathcal{G}_0} = \tilde \a_\infty(u,v)
\]
as $\tr v=0$.
\end{proof}

Thus, we have proved that the operators $\Aa_n$, associated to $\a_n$ converge in the strong resolvent sense to the Bi-Laplacian with Dirichlet boundary conditions on $L^2(\Omega)$. As we have already mentioned, properties of the eigenspace corresponding to the first eigenvalue of the latter operator depend heavily on the geometry of $\Omega$:

If $\Omega$ is a ball (or, in a sense, close enough to a ball), then the first eigenfunction is positive. If $\Omega$ is a square, then the first eigenfunction changes sign. It may also happen, that the first eigenspace is two-dimensional, e.g.\ if $\Omega$ is an annulus whose inner radius is small enough. For all of this, and more, we refer the reader to \cite{Swe01} and the references therein.

We will now prove that the convergence of $\Aa_n$ to $A_\infty$ (at least after passing to a subsequence) entails convergence of the first eigenvalue and the first eigenfunction. It follows that examples of $\Omega$
where the first eigenfunction of the Bi-Laplacian with Dirichlet boundary condition changes sign give rise to examples of domains where the first eigenfunction of our operator also changes sign and thus the associated semigroup is not eventually positive.

In what follows, we write $\lambda_1(\Aa_n)$ for the first eigenvalue of the operator $\Aa_n$. Note, that this eigenvalue can be computed by minimizing the Rayleigh quotient (see Remark \ref{remmay}).
By the monotonicity of the forms $\a_n$, the first eigenvalues are increasing. We will use these facts in the proof of the following

\begin{satz}\label{conva}
For every $n\in \mathbb{N}$, let $\uu_n =(u_1^n, u_2^n)$ be an eigenfunction of $\Aa_n$ for the first eigenvalue $\lambda_1(\Aa_n)$ with $\|\uu_n\|_\Hh=1$. Then there is a subsequence (which, for ease of notation, we index with $n$ again) such that
$\uu_n \rightarrow \uu \in \Hh$ for some $\uu=(u_1, u_2) \in D(\a_\infty)$ and
\[
\lambda_1(\Aa_n)\rightarrow \lambda_1(A_\infty)=\a_\infty(\uu),
\]
i.e.\ $u_1$ is an eigenfunction of $A_\infty$ for the eigenvalue $\lambda_1(A_\infty)$.
\end{satz}

\begin{proof}
We have $\a_n(\uu_n)=\lambda_1(\Aa_n)\leq\lambda_1(A_\infty)$. Thus, since $\gamma_n \geq g_n\nearrow \infty$, we have
\[
\|u_2^n\|^2_\Gamma=\<u_2^n,u_2^n\>_\Gamma \leq \frac{1}{g_n}\a_n[u_n]\leq \frac{1}{g_n}\lambda_1(A_\infty) \longrightarrow 0.
\]
This proves that $u_2^n=\tr u_1^n \rightarrow 0$ in $L^2(\Gamma)$.
Furthermore we have
\[
\|\Delta u_1^n\|_\Omega^2=\<\Delta u_1^n, \Delta u_1^n\> \leq \a(u_n)\leq \lambda_1(A_\infty).
\]
As also $\|u_1^n\|_\Omega^2\leq \|\uu_n\|_\Hh=1$, we can bound the $H^1(\Omega)$-Norm of $u_1^n$. Indeed,
\begin{align*}
\|u_1^n\|^2_{H^1(\Omega)}=\|u_1^n\|_\Omega^2+\<\nabla u_1^n,\nabla u_1^n\>_\Omega=\|u_1^n\|_\Omega^2-\<\Delta u_1^n,u_1^n\>_\Omega\\
\leq 1+\frac{1}{2}\|\Delta u_1^n\|_\Omega^2+\frac{1}{2} \|u_1^n\|_\Omega^2 \leq \frac{3}{2}+\frac{1}{2}\lambda_1(A_\infty).
\end{align*}
By the reflexivity of $H^1(\Omega)$, passing to a subsequence,
we may (and shall) assume that $u_1^n$ converges weakly in
$H^1(\Omega)$ to some $u_1 \in H^1(\Omega)$. As the embedding of $H^1(\Omega)$ into $L^2(\Omega)$ is compact, $u_1^n \rightarrow u_1$ in $L^2(\Omega)$.
Since $(\Delta u_1^n)_n$ is bounded in $L^2(\Omega)$, passing to another subsequence, we obtain $\Delta u_1^n \rightharpoonup w$ for some $w \in L^2(\Omega)$. It follows that for $\varphi \in H^1(\Omega)$
\begin{align*}
\<w,\phi\>_\Omega=\lim_{n\to\infty} \<\Delta u_1^n, \phi\>=\lim_{n\to\infty} \<-\nabla u_1^n, \nabla \phi\>=-\<\nabla u_1, \nabla \phi\>,
\end{align*}
so that  $u_1 \in D(\Delta_N)$ and $\Delta u_1 = w$. Thus, $\Delta u_1^n \rightharpoonup \Delta u_1$.

As the trace is continuous from $H^1(\Omega)$ to $L^2(\Gamma)$,  it is also weakly continuous, so that
 $\tr u_1^n \rightharpoonup \tr u_1.$ Since we know that $\tr u_1^n=u_2^n \rightarrow 0$, we must have $\tr u_1=0$.
Altogether, we have proved that $\uu = (u_1, 0) \in D(\a_\infty)$ and $\uu_n$ converges to $\uu$ in $\Hh$. As the norm is continuous, we find
$\|\uu\|_\Hh=1$.

Since the norm on $L^2(\Omega)$ is weakly lower semicontinuous,
\begin{align*}
\lambda_1(A_\infty) &\leq \a_\infty(\uu)=\|\Delta u_1\|_\Omega^2\\
&\leq \liminf_{n\rightarrow \infty} \|\Delta u_n^1\|_\Omega^2\\
&\leq \liminf_{n\rightarrow \infty} \a_n(\uu_n) \leq  \lambda_1(A_\infty).
\end{align*}
Hence $\lim_{n\to\infty} \a_n(\uu_n)=\lim_{n\to\infty} \lambda_1(\Aa_n)=\lambda(A_\infty)=\a_\infty(\uu)$, proving the claim.
\end{proof}

Combining what was done so far, we obtain

\begin{korollar}\label{sign}
Suppose that the domain $\Omega$ is such that all eigenfunctions of $A_\infty$ for the first eigenvalue $\lambda_1(A_\infty)$ change sign. Then, there is some
$\gamma >0$ such that the operator $\Aa$ on $L^2(\Omega)\times L^2(\Gamma)$ with $\alpha\equiv\beta\equiv 1$ is not eventually positive.
\end{korollar}

\begin{proof}
We consider the sequences $\a_n$, $\Aa_n$ as above and denote by $\uu_n$ a normalized eigenfunction for $\lambda_1(\Aa_n)$. By Theorem \ref{conva}, after passing to a subsequence, $\uu_n$ converges to a function of the form
$\uu = (u_1, 0)$, where $u_1$ is an eigenfunction of $A_\infty$ for $\lambda_1(A_\infty)$ which, by assumption, changes sign. Given a set $S\subseteq\Omega$, we have
\[
|\la \one_S, u_1^n\ra_\Omega - \la \one_S, u_1\ra_\Omega |\leq |\Omega|^{1/2}\,\|u_1^n- u_1\|_\Omega \to 0.
\]
If we now consider sets of the form $S=\{u_1>\eps \}$ and $S=\{u_1<-\eps\}$, we see that for large enough $n$
also $u_1^n$ must change sign whence, for such $n$, the semigroup generated by $-\Aa_n$ cannot be eventually positive.
\end{proof}

\begin{bemerkung}
A concrete example where the first eigenfunction of $A_\infty$ changes sign is given by $\Omega=[0,1]^2$, see
 \cite[Thm 1.1]{Cof82}.
\end{bemerkung}


\begin{thebibliography}{AMPR03}

\bibitem[AF03]{AF03}
Robert~A. Adams and John J.~F. Fournier.
\newblock {\em Sobolev spaces}, volume 140 of {\em Pure and Applied Mathematics
  (Amsterdam)}.
\newblock Elsevier/Academic Press, Amsterdam, second edition, 2003.

\bibitem[AMPR03]{AMPR03}
Wolfgang Arendt, Giorgio Metafune, Diego Pallara, and Silvia Romanelli.
\newblock The {L}aplacian with {W}entzell-{R}obin boundary conditions on spaces
  of continuous functions.
\newblock {\em Semigroup Forum}, 67(2):247--261, 2003.

\bibitem[BHdS20]{BHS20}
Jussi Behrndt, Seppo Hassi, and Henk de~Snoo.
\newblock {\em Boundary value problems, {W}eyl functions, and differential
  operators}, volume 108 of {\em Monographs in Mathematics}.
\newblock Birkh\"{a}user/Springer, Cham, 2020.

\bibitem[Cof82]{Cof82}
Charles~V. Coffman.
\newblock On the structure of solutions {$\Delta ^{2}u=\lambda u$} which
  satisfy the clamped plate conditions on a right angle.
\newblock {\em SIAM J. Math. Anal.}, 13(5):746--757, 1982.

\bibitem[Cos19]{Cos19}
Martin Costabel.
\newblock On the limit {S}obolev regularity for {D}irichlet and {N}eumann
  problems on {L}ipschitz domains.
\newblock {\em Mathematische Nachrichten}, 292(10):2165--2173, 2019.

\bibitem[DG18]{DG18}
Daniel Daners and Jochen Gl\"{u}ck.
\newblock A criterion for the uniform eventual positivity of operator
  semigroups.
\newblock {\em Integral Equations Operator Theory}, 90(4):Paper No. 46, 19,
  2018.

\bibitem[DGK16a]{DGK16a}
Daniel Daners, Jochen Glück, and James~B. Kennedy.
\newblock Eventually and asymptotically positive semigroups on {B}anach
  lattices.
\newblock {\em J. Differential Equations}, 261(5):2607 -- 2649, 2016.

\bibitem[DGK16b]{DGK16b}
Daniel Daners, Jochen Gl\"{u}ck, and James~B. Kennedy.
\newblock Eventually positive semigroups of linear operators.
\newblock {\em J. Math. Anal. Appl.}, 433(2):1561--1593, 2016.

\bibitem[DPZ08]{Denk-Pruess-Zacher08}
Robert Denk, Jan Pr\"{u}ss, and Rico Zacher.
\newblock Maximal {$L_p$}-regularity of parabolic problems with boundary
  dynamics of relaxation type.
\newblock {\em J. Funct. Anal.}, 255(11):3149--3187, 2008.

\bibitem[DT08]{DT08}
Jes\'{u}s~Ildefonso D\'{\i}az and Lourdes Tello.
\newblock On a climate model with a dynamic nonlinear diffusive boundary
  condition.
\newblock {\em Discrete Contin. Dyn. Syst. Ser. S}, 1(2):253--262, 2008.

\bibitem[EF05]{EF05}
Klaus-Jochen Engel and Genni Fragnelli.
\newblock Analyticity of semigroups generated by operators with generalized
  {W}entzell boundary conditions.
\newblock {\em Adv. Differential Equations}, 10(11):1301--1320, 2005.

\bibitem[EPS03]{Escher-Pruess-Simonett03}
Joachim Escher, Jan Pr\"{u}ss, and Gieri Simonett.
\newblock Analytic solutions for a {S}tefan problem with {G}ibbs-{T}homson
  correction.
\newblock {\em J. Reine Angew. Math.}, 563:1--52, 2003.

\bibitem[FGGR08]{FGGR08}
Angelo Favini, Gis\`ele~Ruiz Goldstein, Jerome~A. Goldstein, and Silvia
  Romanelli.
\newblock Fourth order operators with general {W}entzell boundary conditions.
\newblock {\em Rocky Mountain J. Math.}, 38(2):445--460, 2008.

\bibitem[GGS10]{Gazzola-Grunau-Sweers10}
Filippo Gazzola, Hans-Christoph Grunau, and Guido Sweers.
\newblock {\em Polyharmonic boundary value problems}, volume 1991 of {\em
  Lecture Notes in Mathematics}.
\newblock Springer-Verlag, Berlin, 2010.
\newblock Positivity preserving and nonlinear higher order elliptic equations
  in bounded domains.

\bibitem[GM08]{GM08}
Fritz Gesztesy and Marius Mitrea.
\newblock Generalized {R}obin boundary conditions, {R}obin-to-{D}irichlet maps,
  and {K}rein-type resolvent formulas for {S}chr\"{o}dinger operators on
  bounded {L}ipschitz domains.
\newblock In {\em Perspectives in partial differential equations, harmonic
  analysis and applications}, volume~79 of {\em Proc. Sympos. Pure Math.},
  pages 105--173. Amer. Math. Soc., Providence, RI, 2008.

\bibitem[GM11]{GM11}
Fritz Gesztesy and Marius Mitrea.
\newblock A description of all self-adjoint extensions of the {L}aplacian and
  {K}re\u{\i}n-type resolvent formulas on non-smooth domains.
\newblock {\em J. Anal. Math.}, 113:53--172, 2011.

\bibitem[GM20a]{gm20a}
Federica Gregorio and Delio Mugnolo.
\newblock Bi-{L}aplacians on graphs and networks.
\newblock {\em J. Evol. Equ.}, 20(1):191--232, 2020.

\bibitem[GM20b]{gm20b}
Federica Gregorio and Delio Mugnolo.
\newblock Higher-order operators on networks: Hyperbolic and parabolic theory.
\newblock {\em Integr. Eqn. Oper. Theory}, 92(6), 2020.

\bibitem[Gol06]{Gol06}
Gis\`ele~Ruiz Goldstein.
\newblock Derivation and physical interpretation of general boundary
  conditions.
\newblock {\em Adv. Differential Equations}, 11(4):457--480, 2006.

\bibitem[Kat95]{Kat95}
Tosio Kato.
\newblock {\em Perturbation theory for linear operators}.
\newblock Classics in Mathematics. Springer-Verlag, Berlin, 1995.
\newblock Reprint of the 1980 edition.

\bibitem[LM72]{LM72}
Jacques-Louis Lions and Enrico Magenes.
\newblock {\em Non-homogeneous boundary value problems and applications. {V}ol.
  {I}}.
\newblock Springer-Verlag, New York-Heidelberg, 1972.

\bibitem[McL00]{McLean00}
William McLean.
\newblock {\em Strongly elliptic systems and boundary integral equations}.
\newblock Cambridge University Press, Cambridge, 2000.

\bibitem[MO86]{MO86}
Shizuo Miyajima and Noboru Okazawa.
\newblock Generators of positive {$C_0$}-semigroups.
\newblock {\em Pacific J. Math.}, 125(1):161--176, 1986.

\bibitem[Nit10]{Nit10}
Robin Nittka.
\newblock {\em Elliptic and parabolic problems with Robin boundary conditions
  on Lipschitz domains}.
\newblock PhD thesis, Universit{\"a}t Ulm, 2010.

\bibitem[Nit11]{Nit11}
Robin Nittka.
\newblock Regularity of solutions of linear second order elliptic and parabolic
  boundary value problems on {L}ipschitz domains.
\newblock {\em J. Differential Equations}, 251(4-5):860--880, 2011.

\bibitem[Ouh05]{Ouh09}
El~Maati Ouhabaz.
\newblock {\em Analysis of heat equations on domains}, volume~31 of {\em London
  Mathematical Society Monographs Series}.
\newblock Princeton University Press, Princeton, NJ, 2005.

\bibitem[PRZ06]{Pruess-Racke-Zheng06}
Jan Pr\"{u}ss, Reinhard Racke, and Songmu Zheng.
\newblock Maximal regularity and asymptotic behavior of solutions for the
  {C}ahn-{H}illiard equation with dynamic boundary conditions.
\newblock {\em Ann. Mat. Pura Appl. (4)}, 185(4):627--648, 2006.

\bibitem[RZ03]{Racke-Zheng03}
Reinhard Racke and Songmu Zheng.
\newblock The {C}ahn-{H}illiard equation with dynamic boundary conditions.
\newblock {\em Adv. Differential Equations}, 8(1):83--110, 2003.

\bibitem[Sim78]{Sim78}
Barry Simon.
\newblock A canonical decomposition for quadratic forms with applications to
  monotone convergence theorems.
\newblock {\em J. Functional Analysis}, 28(3):377--385, 1978.

\bibitem[Swe01]{Swe01}
Guido Sweers.
\newblock When is the first eigenfunction for the clamped plate equation of
  fixed sign?
\newblock In {\em Proceedings of the {USA}-{C}hile {W}orkshop on {N}onlinear
  {A}nalysis ({V}i\~{n}a del {M}ar-{V}alparaiso, 2000)}, volume~6 of {\em
  Electron. J. Differ. Equ. Conf.}, pages 285--296. Southwest Texas State
  Univ., San Marcos, TX, 2001.

\bibitem[Swe09]{Swe09}
Guido Sweers.
\newblock A survey on boundary conditions for the biharmonic.
\newblock {\em Complex Var. Elliptic Equ.}, 54(2):79--93, 2009.

\bibitem[Tri95]{Tri95}
Hans Triebel.
\newblock {\em Interpolation theory, function spaces, differential operators}.
\newblock Johann Ambrosius Barth, Heidelberg, second edition, 1995.

\bibitem[War13]{War13}
Mahamadi Warma.
\newblock Parabolic and elliptic problems with general {W}entzell boundary
  condition on {L}ipschitz domains.
\newblock {\em Commun. Pure Appl. Anal.}, 12(5):1881--1905, 2013.

\bibitem[Wei80]{Weidmann80}
Joachim Weidmann.
\newblock {\em Linear operators in {H}ilbert spaces}, volume~68 of {\em
  Graduate Texts in Mathematics}.
\newblock Springer-Verlag, New York-Berlin, 1980.

\end{thebibliography}
\end{document}